\documentclass{amsart}

\usepackage{amsmath}
\usepackage{amstext}
\usepackage{amsfonts}
\usepackage{amsthm}
\usepackage{amssymb}
\usepackage{amsrefs}
\usepackage{enumerate}

\numberwithin{equation}{section}

\usepackage{hyperref}
\hypersetup{colorlinks}

\usepackage{color}

\newcommand{\adam}[1]{{\color{blue}{#1}}}
\newcommand{\stig}[1]{{\color{red}{#1}}}

\newcommand{\R}{{\mathbf{R}}}
\newcommand{\E}{{\mathbf{E}}}
\newcommand{\N}{{\mathbf{N}}}
\newcommand{\D}{{\mathcal{D}}}
\newcommand{\F}{{\mathcal{F}}} 

\newcommand{\LB}{{\mathcal{L}}}

\newcommand{\diff}[1]{\,\mathrm{d}#1}

\newcommand{\Cb}{\mathcal{G}_{\mathrm{b}}}
\newcommand{\Cp}{\mathcal{G}_{\mathrm{p}}}
\newcommand{\Cc}{\mathcal{C}}

\DeclareMathOperator{\Tr}{Tr}

\theoremstyle{plain}

\newtheorem{definition}{Definition}[section]
\newtheorem{theorem}[definition]{Theorem}
\newtheorem{lemma}[definition]{Lemma}
\newtheorem{corollary}[definition]{Corollary}
\newtheorem{proposition}[definition]{Proposition}
\newtheorem{assumption}[definition]{Assumption}
\theoremstyle{definition}
\newtheorem{remark}[definition]{Remark}

\begin{document}

\title[Duality in refined Sobolev-Malliavin spaces]
{Duality in refined Sobolev-Malliavin spaces and weak approximation of SPDE}

\author[A.~Andersson]{Adam Andersson}
\address{Adam Andersson\\
Department of Mathematical Sciences\\
Chalmers University of Technology and University of Gothenburg\\
SE-412 96 Gothenburg\\
Sweden}
\email{adam.andersson@chalmers.se}

\author[R.~Kruse]{Raphael Kruse}
\address{Raphael Kruse\\
Technische Universit\"at Berlin\\
Institut f\"ur Mathematik\\
Sek. MA 5-3\\
Stra\ss e des 17.~Juni 136\\
DE-10623 Berlin\\
Germany}
\email{kruse@math.tu-berlin.de}

\author[S.~Larsson]{Stig Larsson}
\address{Stig Larsson\\
Department of Mathematical Sciences\\
Chalmers University of Technology and University of Gothenburg\\
SE-412 96 Gothenburg\\
Sweden}
\email{stig@chalmers.se}

\keywords{SPDE, finite element method, backward Euler, weak
  convergence, convergence of moments, Malliavin calculus, duality,
  spatio-temporal discretization}
\subjclass[2010]{60H15, 60H07, 65C30, 65M60}

\begin{abstract}
  We introduce a new family of refined Sobolev-Malliavin spaces that
  capture the integrability in time of the Malliavin derivative.  We
  consider duality in these spaces and derive a Burkholder type
  inequality in a dual norm.

  The theory we develop allows us to prove weak convergence with
  essentially optimal rate for numerical approximations in space and
  time of semilinear parabolic stochastic evolution equations driven
  by Gaussian additive noise. In particular, we combine a standard Galerkin
  finite element method with backward Euler timestepping. The
  method of proof does not rely on the use of the Kolmogorov equation
  or the It\={o} formula and is therefore non-Markovian in
  nature. Test functions satisfying polynomial growth and mild
  smoothness assumptions are allowed, meaning in particular that we
  prove convergence of arbitrary moments with essentially optimal
  rate.
\end{abstract}

\maketitle

\section{Introduction}
\label{sec1}
The classical Sobolev-Malliavin spaces capture the integrability in
the chance parameter of a random variable and its Malliavin
derivatives. In many situations, where Malliavin calculus is used, in
particular, for stochastic evolution equations, the Malliavin
derivative is a stochastic process. One purpose of this paper is to
introduce a refined family of Sobolev-Malliavin spaces that capture
the integrability properties of the Malliavin derivative with respect
to its time parameter. It turns out that the Malliavin derivative of
the solution to a parabolic stochastic evolution equation has,
depending on the regularity of the noise, good integrability
properties in time and, in the case of trace class noise, it is even
bounded. However, the main purpose of the new feature is not to
measure regularity in a refined way, but to exploit that the
corresponding dual norms are weaker with respect to integrability in
time.

Let $(H,\|\cdot\|,\langle\cdot,\cdot\rangle)$ be a separable Hilbert
space and $Q\in\LB(H)$ be a selfadjoint positive semidefinite linear
operator on $H$. We define the space $H_0=Q^\frac12(H)$ and let
$\LB_2^0=\LB_2(H_0,H)$ be the space of Hilbert-Schmidt operators
from $H_0$ to $H$. We consider a filtered probability space
$(\Omega,\mathcal{F},(\mathcal{F}_t)_{t\in[0,T]},\mathbf{P})$ on which
an $L^2([0,T],H_0)$-isonormal process is defined. For a differentiable
random variable $X$ the Malliavin derivative $DX=(D_tX)_{t\in[0,T]}$
with respect to the isonormal process is an $\LB_2^0$-valued
stochastic process. We introduce, for $p,q\geq2$, the refined
Sobolev-Malliavin spaces $\mathbf{M}^{1,p,q}(H)$ of random variables
$X\in L^2(\Omega,H)$ such that
\begin{align*}
  \|X\|_{\mathbf{M}^{1,p,q}(H)}=\Big(\|X\|_{L^p(\Omega,H)}^p
  +\|DX\|_{L^p(\Omega,L^q([0,T],\LB_2^0))}^p\Big)^\frac1p<\infty.
\end{align*}
The classical Sobolev-Malliavin spaces are obtained for $q=2$. We use
the refined spaces in a duality argument based on the Gelfand triple
\begin{align*}
\mathbf{M}^{1,p,q}(H)\subset L^2(\Omega,H)\subset\mathbf{M}^{1,p,q}(H)^*.
\end{align*}
A key ingredient is the following inequality for the $H$-valued
stochastic It\=o-integral $\int_0^T\Phi\diff{W}$ in the dual norm of
$\mathbf{M}^{1,p,q}(H)$, where $W$ is a cylindrical $Q$-Wiener process
and $\Phi\in L^{p}(\Omega,L^{2}([0,T],\LB_2^0))$ is a predictable
stochastic process. In Theorem~\ref{lemma2:dual3} we show
\begin{align}
  \label{ineq1:Burkholder}
  \Big\|\int_0^T\Phi(t)\diff{W(t)}\Big\|_{\mathbf{M}^{1,p,q}(H)^*}
  \leq\big\|\Phi\big\|_{L^{p'}(\Omega,L^{q'}([0,T],\LB_2^0))},
\end{align}
where $p',q'$ are the conjugate exponents to $p,q \ge 2$. We apply
this inequality in situations, where one usually relies on the
Burkholder-Davis-Gundy inequality, see
Lemma~\ref{lemma2:Burkholder}. There the $L^2(\Omega,H)$-norm of the
stochastic integral is bounded in terms of the
$L^p(\Omega,L^2([0,T],\LB_2^0))$-norm of $\Phi$, whereas here the dual
norm of the integral is bounded by the
$L^{p'}(\Omega,L^{q'}([0,T],\LB_2^0))$-norm of $\Phi$. Since $q'\le
2$, this allows stronger singularities with respect to $t$.

In defining the spaces $\mathbf{M}^{1,p,q}(H)$ some care needs to be
taken. For $q\geq2$ we define the Malliavin derivative on a
non-standard core $\mathcal{S}^q(H)$, see \eqref{def:Sp},
\eqref{eq:SqH}, of smooth and cylindrical random variables, more
regular than in the classical theory in which $q=2$. By proving that
the operator $D\colon\mathcal{S}^q(H)\rightarrow
L^p(\Omega,L^q([0,T],\LB_2^0))$ is well defined and closable, we show
that $\mathbf{M}^{1,p,q}(H)$ are Banach spaces. The proofs are rather
elementary and rely to a large extent on existing results for the case
$q=2$. The spaces are new to the best of our knowledge.

The motivation for introducing the spaces described above is found in
our aim to develop new methods for the analysis of the weak error of
numerical approximations of semilinear parabolic stochastic partial
differential equations of the form
\begin{align}
  \label{eq1:SPDE}
  \diff{X}(t)+AX(t)\diff{t}=F(X(t))\diff{t}+\diff{W(t)},\; t\in(0,T];\quad
  X(0)=X_0.
\end{align}
Both space-time white noise and trace class noise are considered and
the nonlinearity $F$ is allowed to be a Nemytskii operator.  See
Assumption~\ref{as1:A} below for precise conditions on $A$, $F$, $W$,
$X_0$. We treat discretizations in space and time, allowing for any
spatial discretization scheme that satisfies the abstract
Assumption~\ref{as1:Scheme} below. We verify this assumption in
Section~\ref{sec5} for piecewise linear finite element approximations
of the heat equation. Discretization in time is performed by the
semi-implicit backward Euler method. Our main result, weak convergence
of essentially optimal rate, is stated in Theorem~\ref{thm1:main}.

More concretely, our main example is the semilinear stochastic heat
equation, 
\begin{align*}
  \begin{aligned}
  &\dot{u}(\xi,t) -\Delta u (\xi,t) = f(u (\xi,t))+\dot{\eta}(\xi,t),
&&\quad (\xi,t)\in  D\times(0,T],\\ 
&u (\xi,t)=0,
&&\quad (\xi,t)\in  \partial D\times(0,T],\\ 
&u(\xi,0)=u_0(\xi)
&&\quad \xi\in   D, 
  \end{aligned}
\end{align*}
where $f$ is a smooth function with bounded derivatives and
$\dot{\eta}$ is additive noise, white in time and possibly correlated
in space.

Weak convergence for linear stochastic evolution equations was studied
in \cite{Schrodinger}, \cite{debussche2009}, \cite{Geissert},
\cite{larsson2011}, \cite{larsson2013}, \cite{kruse2013},
\cite{LindnerSchilling} and the works \cite{Brehier}, \cite{Brehier2},
\cite{Brehier3}, \cite{hausenblas2003Weak}, \cite{hausenblas2010},
\cite{kopecthesis}*{Chapt.~5}, \cite{Wang}, \cite{Wang2014},
\cite{WangGan} treat semilinear equations with additive noise. Of
these \cite{kopecthesis}*{Chapt.~5} is unique in that it treats a
nonglobal Lipschitz drift term.  In \cite{buckwar2008},
\cite{buckwar2005} the authors study weak convergence for stochastic
ordinary delay differential equations. Most of these works are based
on It\={o}'s formula and Kolmogorov's equation. It becomes apparent
while reading the literature that proving weak convergence of optimal
order is a challenging task. Semilinear equations with multiplicative 
noise were treated in
\cite{AnderssonLarsson}, \cite{conus2014}, \cite{debussche2011}, but
only \cite{conus2014} covers noise more general than linear. No
results are known for multiplicative noise in the form of a nonlinear
Nemytskii operator. As in \cite{Brehier}, \cite{Brehier3}, \cite{conus2014},
\cite{hausenblas2010}, \cite{kopecthesis}*{Chapt.~5}, \cite{Wang}, 
\cite{Wang2014} we allow $F$ to be
a nonlinear Nemytskii operator.

Let $X,Y\in L^2(\Omega,H)$ and $\varphi\colon H\rightarrow\R$ be a sufficiently smooth function of polynomial growth. Our
technique relies on the following linearization of the weak error
\begin{align*}
  \E\big[\varphi(X)-\varphi(Y)\big]=\E\big[ \big\langle
  \tilde{\varphi},X-Y\big\rangle \big],\quad \textrm{where} \quad
  \tilde{\varphi}=\int_0^1\varphi'(\varrho X+(1-\varrho)Y)\diff{\varrho},
\end{align*}
introduced in \cite{KohatsuHiga2} and \cite{kruse2013}. The
  paper \cite{KohatsuHiga2} then proceeds by using an adjoint problem.
  Based on an idea from \cite{kruse2013}, our method is the
following: If $V\subset L^2(\Omega,H)\subset V^*$ is a Gelfand triple
such that $\tilde{\varphi}\in V$, then we obtain by duality
\begin{align*}
  \big|\E\big[\varphi(X)-\varphi(Y)\big]\big|
  \leq \big\|\tilde{\varphi}\big\|_V
  \big\|X-Y\big\|_{V^*}.
\end{align*}
With a good choice of $V$, the error converges in the $V^\ast$-norm
with twice the rate of convergence in the $L^2(\Omega,H)$-norm, which
is the expected rate of weak convergence. For linear equations we
prove that $V=\mathbf{M}^{1,p,p}(H)$ is a good choice for some
$p>2$. The main part of the error $X-Y$ is then a stochastic
convolution $\int_0^T E(T-t)\diff{W(t)}$. Bounding the error operator
$E(T-t)$ in the appropriate norm yields convergence at the price of a
singularity at $t=T$. By using the inequality \eqref{ineq1:Burkholder}
on this integral with sufficiently large $p=q>2$, we may integrate a
stronger singularity and obtain a higher rate of convergence. For
semilinear equations the main difference is that a term involving
$F(X)-F(Y)$ appears. We then use
$V=\mathbf{G}^{1,p}(H)=\mathbf{M}^{1,p,p}(H)\cap L^{2p}(\Omega,H)$. In
Lemma~\ref{lemma5:Lipschitz} we show that $F\colon V^*\rightarrow V^*$
is locally Lipschitz with a constant depending on
$\|X\|_{\mathbf{M}^{1,2p,p}(H)}$,
$\|Y\|_{\mathbf{M}^{1,2p,p}(H)}$. The choice of a stronger $V$-norm is
necessary in order to control the nonlinearity in this way. After
bounding these norms, we may use a standard Gronwall argument to bound
$\|X-Y\|_{V^\ast}$.

As our method does not rely on the use of Kolmogorov's equation or
It\={o}'s formula, it extends to non-Markovian equations. In the work
\cite{AnderssonKovacsLarsson} our method is used to prove weak
convergence for semilinear stochastic Volterra equations driven by
additive noise. Such equations suffer from the lack of a Kolmogorov
equation and therefore the classical proof is not feasible. We hope
that our method will enable weak error analysis for other
non-Markovian equations such as for instance random evolution PDEs. In
this context we mention the work \cite{buckwar2008} in which
non-Markovian stochastic ordinary delay equations with delay in the
diffusion is treated with a completely different method, relying on
an It\={o} formula from the anticipating stochastic
calculus. For a discussion of the difficulties that arise in
connection with a possible extension to multiplicative noise, see
Subsection~\ref{subsec4:3} below.

An additional advantage of the present work is that we only require
the test function $\varphi$ to be twice differentiable, with
derivatives of polynomial growth. This means, in particular, that we
prove convergence of arbitrary moments with the higher rate. Except in
\cite{kruse2013} for the case of linear equations, the test function
in the previous weak error analysis is assumed to have bounded
derivatives and convergence of moments is treated separately, for
example, in \cite{cohen2012}. 

Moreover, the paper \cite{AnderssonKovacsLarsson} demonstrates that the methods
developed in this paper are also applicable to more general test
functionals, 
which not only evaluate the solution at the final time $T$. For example, 
the convergence result in \cite{AnderssonKovacsLarsson} includes covariances of
the form 
\begin{align*}
  \mathrm{Cov}\big( \big\langle X(t_1), \phi_1 \big\rangle, \big\langle X(t_2),
  \phi_2 \big\rangle \big), \quad \phi_1,\phi_2 \in H, \quad t_1,t_2 \in (0,T],
\end{align*}
as admissible test functions. In addition, our weak error estimate
in Theorem~\ref{thm1:main} is uniform over the time interval
  $[0,T]$ unlike earlier results in the literature.

The paper is organized as follows. In Section~\ref{sec2} we present
preliminary material and our basic assumptions on the stochastic
partial differential equation and the numerical scheme. The core of
the paper is Section~\ref{sec3}, which contains our extensions of the
Malliavin calculus. In \ref{subsec3:1} we introduce the refined
Sobolev-Malliavin spaces and prove that they are well defined. Duality
of our new spaces is treated in \ref{subsec3:2}, with the inequality
\eqref{ineq1:Burkholder} and a local Lipschitz bound as the main
results. In \ref{subsec3:3} and \ref{subsec3:4} regularity in terms of
the new spaces is proved for the solution to the stochastic evolution
equation and its approximation, respectively.  Section~\ref{sec4}
contains the weak convergence analysis. In \ref{subsec4:1} we restrict
the discussion to approximations of the stochastic convolution and in
\ref{subsec4:2} we treat semilinear equations. Finally, in
Section~\ref{sec5} we verify our assumption on the numerical method
for a standard finite element approximation of the heat equation.

\section{Setting and preliminaries}
\label{sec2}
\subsection{Analytic preliminaries}
\label{subsec2:1}
Let $(U,\|\cdot\|_U,\langle\cdot,\cdot\rangle_U)$ and
$(V,\|\cdot\|_V,\langle\cdot,\cdot\rangle_V)$ be separable Hilbert
spaces and let $\LB(U,V)$ be the Banach space of all
bounded linear operators $U\rightarrow V$ equipped with the operator norm. 
If $U=V$, then we write
$\LB(U)=\LB(U,U)$ and if $U=H$, we abbreviate $\LB=\LB(H)$. We denote by
$\LB_2(U,V)\subset\LB(U,V)$ the subspace of all Hilbert-Schmidt operators
endowed with the standard norm and inner product
\begin{align*}
  \|T\|_{\LB_2(U,V)}=\Big(\sum_{j\in\N}\|T u_j\|_V^2\Big)^\frac12,
  \quad \langle S,T\rangle_{\LB_2(U,V)}=\sum_{j\in\N}\langle
  Su_j,Tu_j\rangle_{V},
\end{align*}
where both are independent of the particular choice of ON-basis
$(u_j)_{j\in\N}\subset U$.

For separable Hilbert spaces $U_1,\ldots,U_m$, $m \in \N$, we denote
by $\LB^{[m]}(U_1 \times \cdots \times U_m, V)$ the space of
multi-linear operators $b\colon U_1 \times \cdots \times U_m \to
V$. We use the notation $b\cdot(u_1,\dots,u_m)=b(u_1,\dots,u_m)$ for
$u_i \in U_i$, $i =1,\ldots,m$, to emphasize that $b$ is
multi-linear. If $U=U_1 = \ldots = U_m$ we abbreviate $\LB^{[m]}(U
\times \cdots \times U, V) = \LB^{[m]}(U,V)$. The norm
$\|b\|_{\LB^{[m]}(U_1 \times \cdots \times U_m,V)}$ is the smallest
constant $C$ such that
\begin{align}
  \label{def2:Normb}
  \|b\cdot(u_1,\dots,u_m)\|_V
  \leq C\|u_1\|_{U_1} \cdots \|u_m\|_{U_m},\quad \forall u_i\in U_i,\; i =
  1,\ldots,m.
\end{align}

Let $\Cc(U,V)$ denote the space of all continuous mappings $U\to V$ and further let $\Cc_{\mathrm{str}}(U,\LB^{[m]}(U,V))$ be the space of all strongly continuous mappings $U\to \LB^{[m]}(U,V)$, i.e., mappings 
$B\colon U\to \LB^{[m]}(U,V)$, which for $u_1,\dots,u_m$ satisfy that
\begin{align*}
  U\ni x\mapsto B(x)\cdot (u_1,\dots,u_m)\in V
\end{align*}
is continuous.

We next introduce spaces of differentiable mappings. A mapping $\phi\in\Cc(U,V)$ belongs to $\mathcal{G}^m(U,V)$ if the recursively defined G\^ateaux derivatives 
\begin{align*}
&\phi^{(k)}(x)\cdot (u_1,\dots,u_k)\\
&\quad
 =
  \lim_{\epsilon\to0}
  \epsilon ^{-1}
  \big[
    \phi^{(k-1)}
    (x+\epsilon u_k)
    \cdot
    (u_1,\dots,u_{k-1})
    -
    \phi^{(k-1)}
    (x)
    \cdot
    (u_1,\dots,u_{k-1})
  \big],
\end{align*}
exist as limits in $V$ for $k\in\{1,\dots,m\}$,
$x,u_1,\dots,u_k\in U$, and if $\phi^{(k)}(x)\in\LB^{[k]}(U,V)$ are
symmetric $k$-forms for $k\in\{1,\dots,m\}$, $x\in U$, and if
$\phi^{(k)}\in \Cc_{\mathrm{str}}(U,\LB^{[k]}(U,V))$,
$k\in\{1,\dots,m\}$. We remark that if
  $\phi\in\mathcal{G}^m(U,V)$ has continuous derivatives,
  $\phi^{(k)}\in \Cc(U,\LB^{[k]}(U,V))$ for $k\in\{1,\dots,m\}$, then
 it is actually Fr\'echet differentiable, $\phi\in\mathcal{C}^m(U,V)$. 

Let
$\mathcal{C}_{\mathrm{b}}^m(U,V)\subset \Cb^m(U,V)\subset
\mathcal{G}^m(U,V)$
be the subspaces consisting of $\phi$, whose derivatives
$\phi',\dots,\phi^{(m)}$ are bounded (note that $\phi$ needs not be
bounded), and
$\mathcal{C}_{\mathrm{p}}^m(U,V)\subset
\mathcal{G}_\mathrm{p}^m(U,V)\subset \mathcal{G}^m(U,V)$
denotes the analogous space with derivatives of polynomial growth.  On
$\Cb^m(U,V)$ we use the natural seminorm
$|\phi|_{\Cb^m(U,V)}=\sup_{x\in
  U}\|\phi^{(m)}(x)\|_{\LB^{[m]}(U,V)}$.
We define $\Cb^0(U,V)$ to be all bounded continuous mappings
$U\rightarrow V$, endowed with the uniform norm. The first derivative
of $\phi\in\mathcal{G}^1(U,V)$ is an operator
$\phi'(x)\in\LB(U,V)=\LB^{[1]}(U,V)$ for every $x \in U$. When $V=\R$
we may identify $\phi'(x)\in\LB(U,\R)= U^*$ with its gradient
$\phi'(x)\in U$ via $\phi'(x)\cdot u=\langle\phi'(x),u\rangle_U$ by
the Riesz representation theorem. Similarly, for
$\phi\in\mathcal{G}^2(U,\R)$ we will sometimes identify
$\phi''(x)\in\LB^{[2]}(U,\R)$ with an operator $\phi''(x)\in\LB(U)$
via $\phi''(x)\cdot(u_1,u_2)=\langle \phi''(x)u_1,u_2\rangle_U$. For
$\phi \in \Cp^1(U,V)$, the mapping
$[0,1]\ni \rho \mapsto
\tfrac{\mathrm{d}}{\mathrm{d}\rho}\phi(y+\rho(x-y))=\phi'(y+\rho(x-y))\cdot(x-y)\in
V$ is continuous and Bochner integrable and therefore
\begin{align}
  \label{eq2:Taylor2}
    \phi(x) &= \phi(y) +
    \int_{0}^{1} \phi'( y + \rho( x - y) ) \cdot (x - y ) \diff{\rho},
    \quad x,y\in U.
\end{align}

We will use the following version of Gronwall's Lemma, for a proof see
\cite{elliott1992}*{Lemma 7.1}.

\begin{lemma}
  \label{lemma2:Gronwall}
  Let $T>0$, $N\in\N$, $k=\tfrac{T}N$, and $t_n=nk$ for $0\leq
  n\leq N$. If $(\varphi_j)_{j=1}^N$ are nonnegative real
    numbers with
  \begin{align*}
    \varphi_n \leq
    C_1 \,(1 + t_n^{-1+\mu}) + C_2
    \,k\,\sum_{j=0}^{n-1}t_{n-j}^{-1+\nu}\varphi_j, \quad 1\leq n\leq N,
  \end{align*}
  for some constants $C_1,C_2 \ge 0$ and $\mu,\nu>0$, then there exists a
  constant
  $C=C(\mu,\nu,C_2,T)$ such that
  \begin{align*}
    \varphi_n\leq C\,C_1\,(1 + t_n^{-1+\mu}), \quad 1\leq n\leq N.
  \end{align*}
\end{lemma}

We sometimes write $a\lesssim b$ to denote $a\leq C b$ for some
constant $C>0$. Constants arising from the estimates
\eqref{as1:SAnalytic}, \eqref{as1:SAnalytic2}, \eqref{as1:Analyt}, and
\eqref{as1:ErrorOneStep}, as well as trivial numerical constants, will
be suppressed with this symbol.

\subsection{Stochastic preliminaries}
\label{subsec2:2}
Let $(H,\|\cdot\|,\langle\cdot,\cdot\rangle)$ be a separable Hilbert
space and let $Q\in\LB=\LB(H)$ be a selfadjoint, positive semidefinite
operator on $H$ and $Q^\frac12$ its unique positive square root. The
space $H_0=Q^{\frac12}(H)$ is a Hilbert space with scalar product
$\langle u,v\rangle_{H_0}=\langle Q^{-\frac12} u,Q^{-\frac12}
v\rangle$. We denote by $\LB_2^0=\LB_2(H_0,H)$ the space of
Hilbert-Schmidt operators $H_0\rightarrow H$. We consider a filtered
probability space
$(\Omega,\mathcal{F},(\mathcal{F}_t)_{t\in[0,T]},\mathbf{P})$ and the
corresponding Bochner spaces $L^p(\Omega,
V)=L^p((\Omega,\mathcal{F},\mathbf{P}),V)$, $p\in[1,\infty]$, $V$ a
Banach space. We abbreviate
$L^2(\Omega)=L^2(\Omega,\R)$.
We assume that
$(W(t))_{t\in[0,T]}$ is a cylindrical $Q$-Wiener process, meaning that
$W\in \mathcal{C}([0,T],\LB(H_0,L^2(\Omega)))$ is such that $t\mapsto W(t)u$ is
an $\mathcal{F}_t$-predictable real-valued Brownian motion for every
$u\in H_0$ and
\begin{align*}
\E\big[W(s)u\, W(t)v\big]=\min(s,t)\langle u,v\rangle_{H_0},
\quad u,v\in H_0,\ s,t\in[0,T].
\end{align*}
For predictable $\Phi\in L^2([0,T]\times\Omega,\LB_2^0)$ the $H$-valued
stochastic It\=o-integral
\begin{align*}
  \int_0^T\Phi(t)\diff{W(t)}\in L^2(\Omega,H),
\end{align*}
is a well defined random variable. For details on the construction of
cylindrical Wiener processes and the corresponding stochastic integral
we refer to \cites{daprato1992, roeckner2007, UMD}.  For technical
reasons we assume that the $\sigma$-field $\F$ is generated by
$(W(t))_{t \in [0,T]}$ and the filtration
$(\mathcal{F}_t)_{t\in[0,T]}$ is the natural filtration associated
with $(W(t))_{t \in [0,T]}$.

We cite the following special case of Burkholder's inequality
\cite{daprato1992}*{Lemma 7.2}.
\begin{lemma}
  \label{lemma2:Burkholder}
  Let $(\Phi(t))_{t\in[0,T]}$ be a predictable and $\LB_2^0$-valued process
  such that $\|\Phi\|_{L^p(\Omega,L^2([0,T],\LB_2^0))} < \infty$ for some $p
  \ge 2$. Then there exists a constant $C_p$, such that
  \begin{equation*}
    \Big\|\int_0^T\Phi(s)\diff{W(s)}\Big\|_{L^p(\Omega,H)}\leq
    C_p\|\Phi\|_{L^p(\Omega,L^2([0,T],\LB_2^0))}.
  \end{equation*}
\end{lemma}

\subsection{The stochastic equation}
\label{subsec2:3}
We study equation \eqref{eq1:SPDE} under the following assumption and
recall that the solution $X$ takes values in $H$.
\begin{assumption}
  \label{as1:A}
  \begin{enumerate}[\upshape (i)]
    \item
      Let $(A,\D(A))$ be a linear operator on $H$ such that $A^{-1}\in\LB(H)$
      exists and $-A$ is the generator of an analytic semigroup
      $(S(t))_{t\geq0}$ of bounded linear operators $S(t)=\mathrm{e}^{-tA}$ on
      $H$.
    \item The initial value $X_0$ is
      deterministic and satisfies $X_0\in \dot{H}^{2\beta}$, for some
      $\beta\in(0,1]$, where $\dot{H}^{\alpha}\subset H$ denotes the
      domain of $A^{\frac\alpha2}$.
    \item
      The covariance operator $Q$ satisfies
      $\|A^{\frac{\beta-1}2}\|_{\LB_2^0}
      =\|A^{\frac{\beta-1}2}Q^\frac12\|_{\LB_2}<\infty$, for the same
      $\beta$ as in {\upshape(ii)}.
    \item The drift $F\colon H\rightarrow H$ is assumed to be twice
      differentiable in the sense $F\in\Cb^1(H,H)\cap\Cb^2(H,\dot{H}^{-1})$,
      where $\dot{H}^{-1}$ is defined below.
  \end{enumerate}
\end{assumption}

Under Assumption~\ref{as1:A} (i) the fractional powers $A^{\frac{r}2}$
for $r\in\R$ are well defined, see \cite{pazy1983}*{Section 2.6}. We
define the norms $\|v\|_r =\|A^{\frac{r}2}v\|$ and let
$\dot{H}^r=\D(A^\frac{r}2)$ for $r\geq0$. For $r<0$ we define
$\dot{H}^r$ as the closure of $H$ under the norm $\|v\|_r$. The spaces
$\dot{H}^r\subset H\subset\dot{H}^{-r}$ form a Gelfand triple for
$r>0$.

The analytic semigroup $(S(t))_{t\ge0}$ generated by $-A$
satisfies, see \cite{pazy1983}*{Section 2.6},
\begin{alignat}{2}
  \label{as1:SAnalytic}
  \|A^\varrho S(t)\|_{\LB}
  &\leq C_\varrho t^{-\varrho}, &
  \quad &t>0,\ \varrho\geq0,\\
  \label{as1:SAnalytic2}
  \|(S(t)-I)A^{-\varrho}\|_{\LB}
  &\leq C_\varrho t^{\varrho}, &
  \quad &t\ge0,\ 0<\varrho\leq1.
\end{alignat}

Under Assumption~\ref{as1:A}, the stochastic equation
\eqref{eq1:SPDE} has a mild solution $X\in
\Cc([0,T],L^p(\Omega,H))$, for every $p\geq2$, in the
sense that it satisfies the integral equation
\begin{equation}
  \label{eq1:SPDEmild}
    X(t)=S(t)X_0+\int_0^tS(t-s)F(X(s))\diff{s}+\int_0^tS(t-s)\diff{W(s)},\quad
    t\in[0,T],
\end{equation}
and
\begin{align}
  &\label{ineq2:moment}\sup_{t\in[0,T]}\|X(t)\|_{L^p(\Omega,H)}\leq
  C(1+\|X_0\|).
\end{align}
For every $\gamma \in [0,\beta)$ the solution satisfies $X(t)\in
\dot{H}^\gamma$, $\mathbf{P}$-a.s., for all $t\in[0,T]$. For more
details we refer to \cite{daprato1992}, \cite{jentzen2010b},
\cite{kl2010a}, and the references therein.

In \cite{AnderssonLarsson} and \cite{debussche2011} the authors assume
$F\in\mathcal{C}_{\mathrm{b}}^2(H,H)$, which works well for the analysis but has the
following disadvantage: If $D\subset \R^d$, $d=1,2,3$, $H=L^2(D)$ and
$F\colon H\rightarrow H$ is a Nemytskii operator, i.e., a mapping in
the form $g\mapsto F(g)=f(g(\cdot))$, where $f\in\mathcal{C}_{\mathrm{b}}^2(\R,\R)$, then in general $F\not\in \mathcal{C}_{\mathrm{b}}^1(H,H)$. This
disqualifies the most interesting examples of nonlinearities $F$. On
the other hand $F\in \Cb^1(H,H)$ and by the Sobolev embedding theorem
$F\in\Cb^2(H,\dot{H}^{-\frac{d}2+\epsilon})$ for $\epsilon>0$ and
hence Assumption~\ref{as1:A} admits Nemytskii operators for $d=1$. See
\cite{Wang}*{Example 5.1} for a verification. For $d=2,3$ one needs to
assume $F\in\Cb^2(H,\dot{H}^{-s})$ with $s>1$, which works for
spectral Galerkin approximations but not for the finite element method
due to the restriction on $\varrho$ in \eqref{as1:AhPhA} below. In
\cite{AnderssonKovacsLarsson} this restriction is removed, allowing
for finite element discretization also for $d=2,3$. Papers that
include Nemytskii operators are \cite{Brehier}, \cite{Brehier3},
\cite{hausenblas2010}, \cite{Wang}, \cite{Wang2014} and our
Assumption~\ref{as1:A} (iv) is a reformulation of
\cite{Wang}*{Assumption 5.1}.

\subsection{Approximation of the solution}
\label{subsec2:4}
We approximate equation \eqref{eq1:SPDE} in finite-dimensional
approximation spaces $V_{h}\subseteq H$, ${h}\in(0,1]$. The
parameter ${h}\in(0,1]$ is a refinement parameter. We denote by
$P_{h}\colon H \to V_{h}$ the orthogonal projector onto $V_{h}$ and by
$(A_{h})_{{h}\in(0,1]}$ a family of operators $A_{h}\colon
V_{h}\rightarrow V_{h}$ approximating $A$. The assumptions on
$(V_{h})_{{h} \in (0,1]}$, and $(A_{h})_{{h}\in(0,1]}$ are given in
Assumption~\ref{as1:Scheme} below.

For the time discretization let $k \in (0,1)$ be the constant step
size. We define the discrete time points by $t_n=nk$, $n=0,\dots,N$,
where $N = N(k) \in \N$ is determined by $t_N \le T < t_{N} + k$.  We
define the operator $S_{{h},k}=(I+kA_{{h}})^{-1}P_h$ and notice that
$S_{h,k} Q^\frac12 \in \LB_2(H)$, since $S_{h,k}$ is a finite rank
operator. Hence, it is a valid integrand for the stochastic
integral. Our completely discrete scheme is to find the recursive
sequence $(X_{{h},k}^n)_{n=0}^N \subset V_h$ given by the
semi-implicit Euler-Maruyama method:
\begin{equation}
  \label{eq1:SPDEiterat}
   \begin{split}
    &X_{{h},k}^{n+1}=S_{{h},k}X_{{h},k}^n+kS_{{h},k}F(X_{{h},k}^n)
    +\int_{t_n}^{t_{n+1}}S_{{h},k}\diff{W(s)},\quad n=0,\dots,N-1;\\
    &X_{{h},k}^0=P_{h} X_0.
  \end{split}
\end{equation}
By iterating \eqref{eq1:SPDEiterat} we obtain the discrete analog of
\eqref{eq1:SPDEmild}
\begin{equation}
  \label{eq3:SPDEfulldisc}
  \begin{split}
    X_{{h},k}^n&=S_{{h},k}^nP_{h} X_0+k\sum_{j=0}^{n-1}S_{{h},k}^{n-j}
    F(X_{{h},k}^j)
\\ & \quad +\sum_{j=0}^{n-1}\int_{t_j}^{t_{j+1}}
    S_{{h},k}^{n-j}\diff{W(t)},\quad n=0,\dots,N.
  \end{split}
\end{equation}
Further, we define the error operators $E_{{h},k}^n$, $h, k \in (0,1]$, by
\begin{align}
 \label{def1:ErrorOp}
E_{{h},k}^n := S(nk)- S_{{h},k}^n.
\end{align}
We now state our assumption on the numerical discretization.

\begin{assumption}
  \label{as1:Scheme}
  The linear operators $A_{h}\colon V_{h} \rightarrow V_{h}$ and the
  orthogonal projectors $P_h \colon H \to V_h$, ${h}\in(0,1]$,
  satisfy
  \begin{align}
    \label{as1:Analyt}
    \|A_h^\varrho S_{{h},k}^n  \|_{\LB} &\leq C t^{-\varrho}_n,\quad
    n = 1,\ldots,N,\quad \varrho\ge0,\\
    \label{as1:AhPhA}
    \| A_h^{-\varrho} P_h A^{\varrho} \|_{\LB} &\leq C,
   \quad 0 \le \varrho \le \tfrac12,
  \end{align}
  uniformly in $h,k \in (0,1]$,
  and, for $0\leq\theta\leq 2$, $-\theta\leq\varrho\leq\min(1,2-\theta)$,
  \begin{align}
    \label{as1:ErrorOneStep}
    \| E_{{h},k}^n {A^\frac\varrho2}\|_{\LB} \leq C\big(h^{\theta}
    +k^{\frac\theta2}\big) t_n^{-\frac{\theta+\varrho}2},
    \quad n = 1,\ldots,N.
  \end{align}
\end{assumption}

We emphasize that the restriction $\varrho\le\frac12$ in
\eqref{as1:AhPhA} is dictated by our desire to include standard finite
element spaces, for which $V_h\subset \dot{H}^{1}$, and no better.
We remark that the error estimate \eqref{as1:ErrorOneStep} is
non-standard, due to the low regularity regime we consider. In fact,
when $\varrho\ge0$, it corresponds to an error estimate for the
deterministic linear equation with rough initial data, i.e.,
$S(t)X_0=S(t) A^{\frac\varrho2}x$ with $x\in H$, so that
$X_0=A^{\frac\varrho2}x\in \dot{H}^{-\varrho}$. We verify
\eqref{as1:ErrorOneStep} in Section~\ref{sec5} for the finite element
method and the heat equation by means of interpolation techniques,
using already established results from \cites{kruse2012,kruse2013}. By
\cite{kruse2012}*{Example 3.4}, spectral Galerkin approximations also
fit under our Assumption~\ref{as1:Scheme}.

Finally, for future reference, we formulate an important consequence
of the smoothing properties \eqref{as1:SAnalytic} and
\eqref{as1:Analyt}, \eqref{as1:AhPhA}, respectively, in conjunction
with the assumption on the covariance operator in
Assumption~\ref{as1:A}~(iii).

\begin{lemma} \label{lemma:Slqbound} Let Assumptions~\ref{as1:A} and
  \ref{as1:Scheme} hold with $\beta\in[0,1]$. Let
  $q\in[2,\frac{2}{1-\beta})$ with $q=\infty$ allowed if
  $\beta=1$. Then
\begin{align*}
  \|S\|_{L^q([0,T],\LB_2^0)}
  \le C \|A^{\frac{\beta-1}2}\|_{\LB_2^0}
\end{align*}
and
\begin{align*}
  \Big(k\sum_{j = 1 }^{N} \| S_{h,k}^{j}\|_{\LB_2^0}^q\Big)^{1/q}
  \le C \|A^{\frac{\beta-1}2}\|_{\LB_2^0}.
\end{align*}
\end{lemma}

\begin{proof}
  Let first $q<\infty$.  By \eqref{as1:SAnalytic} with
  $\varrho=\frac{1-\beta}{2}$ we get
  \begin{align*}
    \|S\|_{L^q([0,T],\LB_2^0)}^q
    &= \int_0^T \|S(t)\|_{\LB_2^0}^q \diff{t}
    \le \int_0^T \|A^{\frac{1-\beta}{2}}S(t)\|_{\LB}^q \diff{t} \,
    \|A^{\frac{\beta-1}2}\|_{\LB_2^0}^q
    \\ &
    \le C\int_0^T t^{-q\frac{1-\beta}{2}} \diff{t} \,
    \|A^{\frac{\beta-1}2}\|_{\LB_2^0}^q
    \le C \|A^{\frac{\beta-1}2}\|_{\LB_2^0}^q .
  \end{align*}
  For the second inequality we use instead \eqref{as1:Analyt},
  \eqref{as1:AhPhA} with $\varrho=\frac{1-\beta}{2}$ to get
  \begin{align*}
    \| S_{h,k}^{j}\|_{\LB_2^0}
    &\le \| A_h^{\frac{1-\beta}{2}} S_{h,k}^{j}\|_{\LB} \|A_h^{\frac{\beta-1}2}P_h\|_{\LB_2^0}
    \\ &
     \le \| A_h^{\frac{1-\beta}{2}} S_{h,k}^{j}\|_{\LB}
    \|A_h^{\frac{\beta-1}2}P_h A^{\frac{1-\beta}2}\|_{\LB}\|A^{\frac{\beta-1}2}\|_{\LB_2^0}
    \\ &
    \le C t_j^{-\frac{1-\beta}{2}}
   \|A^{\frac{\beta-1}2}\|_{\LB_2^0},
  \end{align*}
which can be summed as desired.  The case when $q=\infty$, $\beta=1$
is now obvious.
\end{proof}

\section{Malliavin calculus}
\label{sec3}

The papers \cite{grorud1992} and \cite{leon1998} are the earliest
works to treat Malliavin calculus for stochastic evolution equations
in the Hilbert space framework. Later it was used in several papers
related to optimal control of stochastic partial differential
equations, in particular, in connection with backward stochastic
differential equations \cite{FuhrmanTessitore} and backward stochastic
Volterra integral equations in Hilbert spaces \cite{BSVIE}. Malliavin
differentiability of solutions to stochastic evolution equations is
proved in \cite{FuhrmanTessitore}. There are also works using the
Malliavin calculus for specific equations outside the setting of the
present paper and it is more extensively developed for equations
studied in the framework of \cite{walsh1986}, see the book
\cite{SanzSole}. We mention also the papers \cite{AnderssonLarsson},
\cite{Brehier}, \cite{Brehier2}, \cite{buckwar2008},
\cite{KohatsuHiga2}, \cite{debussche2011},
\cite{hausenblas2003Weak}, \cite{hausenblas2010}, \cite{KohatsuHiga1},
\cite{kopecthesis}*{Chapt.~5}, \cite{WangGan}, where the Malliavin calculus is applied to the problem
of proving weak convergence. Below we take a new direction and
introduce in Subsection~\ref{subsec4:1} a family of refined
Sobolev-Malliavin spaces. We show in Subsection~\ref{subsec4:2} that
these spaces are particularly useful in connection with duality.

\subsection{Refined Sobolev-Malliavin spaces}
\label{subsec3:1}
Let $I\colon L^2([0,T],H_0)\rightarrow L^2(\Omega)$ be the mapping given by
\begin{align*}
I(\phi)=\int_0^T \phi(t)\diff{W(t)},\quad\phi\in L^2([0,T],H_0),
\end{align*}
where we identify $L^2([0,T],H_0)\cong L^2([0,T],\LB_2(H_0,\R))$. This
identification is important since an $\R$-valued stochastic integral has an
$L^2([0,T],\LB_2(H_0,\R))$-valued integrand. Fix an ON-basis
$(\phi_j)_{j \in \N} \subset L^2([0,T],H_0)$, let $\mathcal{P}_n$ be the
set of random variables given by $n$:th order
polynomials of the random variables $(I(\phi_j))_{j\in\N}$. The set
$\mathcal{P}=\cup_{n\in\N}\mathcal{P}_n$ is independent of the choice of basis,
see \cite{Janson}, and
\begin{align}
  \label{WienerChaos}
  \mathcal{P} \subset L^p(\Omega)\  \textrm{is dense
  for}\
  1\leq p<\infty.
\end{align}

Let $2\leq q\leq\infty$ and let the mapping $i\colon L^q([0,T],H_0)\rightarrow
L^2([0,T],H_0)$ denote the canonical embedding. Let $\mathcal{S}^q$ be the
set of random variables $F$ of the form
\begin{align}
  \label{def:Sp}
  \begin{aligned}
  F&=f(I(i(\phi_1)),\dots,I(i(\phi_n))), \\
  f&\in
  \mathcal{C}_{\mathrm{p}}^1(\R^n,\R ),\  (\phi_j)_{j=1}^n\subset
  L^q([0,T],H_0),\ n\in\N.
  \end{aligned}
\end{align}
The class $\mathcal{S}^2$ is standard in Malliavin calculus and is usually
denoted by $\mathcal{S}$. Our definition coincides with that in \cite{kruse2013}
but in the standard work \cite{nualart2006} and many other works
$\mathcal{C}_\mathrm{p}^\infty(\R^n,\R)$ is used instead of
$\mathcal{C}_\mathrm{p}^1(\R^n,\R)$. The classes $\mathcal{S}^q$ for
$q>2$ are new to our knowledge.

\begin{lemma}
  \label{lemma3:dense}
  For $1\leq p<\infty$ and $2\leq q\leq\infty$, $\mathcal{S}^{q}\subset
  L^p(\Omega)$ is dense.
\end{lemma}

\begin{proof}
  Without causing confusion we also let $i$ denote the canonical embedding from
  $L^q([0,T],\R)$ to $L^2([0,T],\R)$. We notice the isomorphism
  $L^2([0,T],H_0)\cong L^2([0,T],\R)\otimes H_0$.

  Since there even exists a bounded ON-basis of the space $L^2([0,T],\R)$
  we clearly find a sequence $(f_n)_{n\in\N}\subset L^q([0,T],\R)$ such that
   $(i(f_n))_{n\in\N}$ is an ON-basis for $L^2([0,T],\R)$.
  If $(h_n)_{n\in\N}$ is an ON-basis for $H_0$, then $(i(f_m)\otimes
  h_n)_{m,n\in\N}$ is an ON-basis for
  $L^2([0,T],\R)\otimes H_0$. In particular, we have that $i(f_m\otimes h_n)=
  i(f_m)\otimes h_n$.

  Since the result \eqref{WienerChaos} is independent of the choice of the
  basis, we conclude our assertion by using the sequence $(I(i(f_m\otimes
  h_n)))_{m,n\in\N}$.
\end{proof}

For $1\leq p<\infty$ and $2\leq q\leq\infty$ we
define the action of the Malliavin derivative $D\colon
\mathcal{S}^{q}\rightarrow L^p(\Omega,L^q([0,T],H_0))$ on a random variable $F$
of the form \eqref{def:Sp} by
\begin{align*}
D_t\,F=\sum_{j=1}^n\partial_j f(I(i(\phi_1)),\dots,I(i(\phi_n)))\otimes\phi_j(t), \quad t\in[0,T].
\end{align*}
This is well defined because $\phi_1,\dots,\phi_n\in L^q([0,T],H_0)$, the
random variables $I(\phi_1),\dots,I(\phi_n)$ are Gaussian with all existing
moments and since $f$ has polynomial growth. By a direct modification of
\cite{kruse2013}*{Proposition 4.2} it does not depend on the specific
representation of $F$.

We remark that for $q = 2$ the linear operator $D \colon
\mathcal{S}^{2} \to L^p(\Omega,L^2([0,T],H_0))$ is the standard
Malliavin derivative. Technically speaking, we have restricted the
domain of the Malliavin derivative to $\mathcal{S}^q \subset
\mathcal{S}^2$ for $2<q \leq\infty$. By this we have ensured that
$D|_{\mathcal{S}^q}$ maps into the smaller space
$L^p(\Omega,L^q([0,T],H_0)) \subset L^p(\Omega,L^2([0,T],H_0))$.

We define the Malliavin derivative for $H$-valued random variables as
in \cite{kruse2013}*{Chapt.~4}, \cite{nualart2006}*{Chapt.~1}. For
this we denote by $\mathcal{S}^{q}(H)$ the collection of all
$H$-valued smooth random variables of the form
\begin{align}
\label{eq:SqH}
  X=\sum_{j=1}^n h_j\otimes F_j,\quad h_1,\dots, h_n\in H,\quad
  F_1,\dots,F_n\in\mathcal{S}^q,\ n\in\N.
\end{align}
Since $H$ is separable and by Lemma~\ref{lemma3:dense} it follows that
$\mathcal{S}^q(H)$ is dense in $L^p(\Omega,H)$ for all $1 \le p < \infty$.
The Malliavin derivative $D\colon\mathcal{S}^q(H)\rightarrow
L^p(\Omega,L^q([0,T],\LB_2^0))$ acts in the following way:
\begin{align*}
  D_tX=D_t\sum_{j=1}^n h_j\otimes F_j=\sum_{j=1}^n h_j\otimes D_tF_j, \quad
  t\in[0,T].
\end{align*}
Here we did the identifications
\begin{align*}
H\otimes L^p(\Omega,L^q([0,T],H_0)) \cong L^p(\Omega,H \otimes L^q([0,T],H_0))
\cong L^p(\Omega,L^q([0,T],\LB_2^0)).
\end{align*}
We write $D_t^uX=D_tXu\in L^2(\Omega,H)$ for the derivative in the direction
$u\in H_0$.

In the final step of its construction we extend the domain of the Malliavin
derivative to its closure with respect to the graph norm. For this we recall
that an unbounded operator $A\colon U\rightarrow V$ is \emph{closable} if and
only if for every $(u_n)_{n\in\N}\subset U$ such that
$\lim_{n\rightarrow\infty}u_n=0$ and $\lim_{n\rightarrow\infty}Au_n=v$, we have
$v=0$.
\begin{lemma}
  \label{lemma3:closable}
  The Malliavin derivative $D\colon\mathcal{S}^q(H)\rightarrow
  L^p(\Omega,L^q([0,T],\LB_2^0))$ is closable for $1<p<\infty$
  and $2\leq q\leq\infty$.
\end{lemma}

\begin{proof}
  We will use the fact that $D\colon \mathcal{S}^2(H)\rightarrow
  L^p(\Omega,L^2([0,T],\LB_2^0))$ is closable for $p>1$,
  \cite{kruse2013}*{Proposition 4.4}. Let $(X_n)_{n\in\N} \subset
  \mathcal{S}^q(H) \subset \mathcal{S}^2(H)$ be a sequence satisfying
  $\lim_{n\rightarrow\infty} X_n=0$ in $L^p(\Omega,H)$ such that
  $\lim_{n\rightarrow\infty} DX_n= Z$ in
  $L^p(\Omega,L^q([0,T],\LB_2^0))$ and hence also in
  $L^p(\Omega,L^2([0,T],\LB_2^0))$. By the closability we have $Z=0$
  in $L^p(\Omega,L^2([0,T],\LB_2^0))$ and hence also in
  $L^p(\Omega,L^q([0,T],\LB_2^0))$.
\end{proof}

For $1<p<\infty$ and $2\leq q\leq\infty$ we can therefore consider the closure
$\mathbf{M}^{1,p,q}(H)$ of $\mathcal{S}^q(H)$ with respect to the norm
\begin{align*}
  \|X\|_{\mathbf{M}^{1,p,q}(H)}=\Big(\|X\|_{L^p(\Omega,H)}^p
  +\|DX\|_{L^p(\Omega,L^q([0,T],\LB_2^0))}^p\Big)^\frac1p.
\end{align*}
Clearly, the spaces $\mathbf{M}^{1,p,2}(H)$, $p>1$, coincide with the classical
Sobolev-Malliavin spaces of the Malliavin calculus, which are usually denoted by
$\mathbf{D}^{1,p}(H)$. The standard Malliavin derivative is uniquely
extended to an operator from $\mathbf{M}^{1,p,2}(H)$ to
$L^p(\Omega,L^2([0,T],\LB_2^0))$. In addition it holds $\mathbf{M}^{1,p,q_1}(H)
\subset \mathbf{M}^{1,p,q_2}(H)$ for all $\infty\ge q_1 \ge q_2 \ge 2$ and from Lemma
\ref{lemma3:closable} it follows that the restriction of the standard Malliavin
derivative $D|_{\mathbf{M}^{1,p,q}(H)}$ is a well-defined operator from
$\mathbf{M}^{1,p,q}(H)$ to $L^p(\Omega,L^q([0,T],\LB_2^0))$.
If $p = q$, we abbreviate $\mathbf{M}^{1,p}(H)=\mathbf{M}^{1,p,p}(H)$.

The space $\mathbf{M}^{1,2}(H)$ is a Hilbert space
and it has a well developed theory of Malliavin calculus. The adjoint of the
Malliavin derivative $D\colon \mathbf{M}^{1,2}(H) \subset
L^2(\Omega,H)\rightarrow
L^2([0,T]\times\Omega,\LB_2^0)$ is called the \emph{divergence} operator or the
\emph{Skorohod} integral and is denoted by $\delta\colon
L^2([0,T]\times\Omega,\LB_2^0)\rightarrow L^2(\Omega,H)$ with domain
$\D(\delta)$. The duality reads
\begin{align}
  \label{eq2:IBP}
  \big\langle X,\delta\Phi\big\rangle_{L^2(\Omega,H)}= \big\langle
  DX,\Phi\big\rangle_{L^2([0,T]\times\Omega,\LB_2^0)}, \quad
  X \in \mathbf{M}^{1,2}(H),\ \Phi\in\D(\delta).
\end{align}
We refer to this as the Malliavin integration by parts formula. It is well
known that for predictable $\Phi\in\D(\delta)$ the action of $\delta$ coincides
with that of the $H$-valued It\={o} integral, i.e.,
$\delta\Phi=\int_0^T\Phi(t)\diff{W(t)}$, \cite{kruse2013}*{Proposition 4.12}.

In the remainder of this subsection we state a modification of the chain rule
from \cite{kruse2013}*{Lemma 4.7} and a product rule for the Malliavin
derivative.

\begin{lemma}
  \label{lemma3:chainrule}
  Let $U, V$ be two separable Hilbert spaces and let $\gamma
    \in\Cp^1(U,V)$, be such that there
  exist constants $C$ and $r \ge 0$ with
  \begin{align*}
    \| \gamma(u) \|_V \le C \big( 1 + \| u \|^{1+r}_U \big), \quad \| \gamma'(u)
    \|_{\LB(U,V)} \le  C \big( 1 + \| u \|^{r}_U \big),
  \end{align*}
  for all $u \in U$. Then, for $1<p<\infty$, $2\leq q \leq\infty$ and $X
  \in \mathbf{M}^{1,(1+r)p,q}(U)$, it follows that $\gamma(X) \in
  \mathbf{M}^{1,p,q}(V)$ with $\|  \gamma(X) \|_{\mathbf{M}^{1,p,q}(V)}
  \lesssim ( 1 + \| X \|_{\mathbf{M}^{1,(1+r)p,q}(U)}^{1+r} )$ and
  \begin{align}
    \label{Chain1}
    D_t(\gamma (X))&=\gamma'(X)\cdot D_t X,\quad t \in [0,T].
  \end{align}
\end{lemma}

\begin{proof}
  Let $p > 1$ be arbitrary. For $q=2$ the result follows directly from
  \cite{kruse2013}*{Lemma 4.7}. In fact \cite{kruse2013}*{Lemma 4.7} is stated for Fr\'echet differentiable $\gamma$ but checking the proof line by line shows that $\gamma\in\Cp^1(U,V)$ is sufficient.
  From this, it suffices to show that $\| \gamma(X) \|_{\mathbf{M}^{1,p,q}(V)}
  < \infty$ if $X \in \mathbf{M}^{1,(1+r)p,q}(U)$ for $q > 2$. Indeed, from the
  polynomial growth condition it follows that
  \begin{align*}
    \big\| \gamma(X) \big\|_{L^p(\Omega,V)} \le C \big( 1 + \| X
    \|_{L^{(1+r)p}(\Omega,U)}^{1+r} \big) \le C \big( 1 + \| X
    \|_{\mathbf{M}^{1,(1+r)p,q}(U)}^{1+r} \big).
  \end{align*}
  Moreover, it holds
  \begin{align*}
    &\big\| D \gamma(X) \big\|_{L^p(\Omega, L^q( [0,T],\LB_2(H_0,V)))}\\
    &\quad = \big( \E \big[ \big\| \gamma'(X)\cdot D X \big\|_{L^q(
    [0,T],\LB_2(H_0,V))}^p \big] \big)^{\frac{1}{p}}\\
    &\quad \lesssim \big( \E \big[ \big( 1 +  \big\|  X \big\|^r_U \big)^p \|
    DX \|_{L^q( [0,T],\LB_2(H_0,U))}^p \big] \big)^{\frac{1}{p}}\\
    &\quad \le \big( 1 + \| X \|^{r}_{L^{(1+r)p}(\Omega,U)} \big) \big\| DX
    \big\|_{L^{(1+r)p}(\Omega,L^q([0,T],\LB_2(H_0,U)))}\\
    &\quad \lesssim \big( 1 + \| X \|^{1+r}_{\mathbf{M}^{1,(1+r)p,q}(U)} \big),
  \end{align*}
  where we applied the polynomial growth condition on $\gamma'$ and
  H\"older's inequality with exponents $(r+1)/r$ and $r+1$.  This
  completes the proof.
\end{proof}

\begin{lemma}
  \label{lemma3:prodrule}
  Let $U_1,U_2,V$ be separable Hilbert spaces and $1 < p <\infty$, $2 \leq
  q\leq\infty$. For $\sigma \in \Cb^0(U_1,\LB(U_2,V)) \cap
  \Cb^1(U_1,\LB(U_2,V))$ and
  $X \in \mathbf{M}^{1,2p,q}(U_1)$ and $Y \in
  \mathbf{M}^{1,2p,q}(U_2)$ it holds $\sigma(X) \cdot Y \in \mathbf{M}^{1,p,q}(V)$. 
  In addition, we have
  \begin{align}
    \label{Product}
    D_t (\sigma(X) \cdot Y)&= \sigma'(X)\cdot (D_t X, Y) + \sigma(X) \cdot
    D_t Y, \quad t \in [0,T].
  \end{align}
\end{lemma}

\begin{proof}
  The proof is done by an application of the chain rule. For this define the
  mapping $\gamma \colon U_1 \times U_2 \to V$ given by $\gamma(x,y)= \sigma(x)
  \cdot y$. Certainly, it holds $\gamma \in \Cp^1( U_1 \times U_2, V)$
  and we have $\| \gamma(x,y)\|_V = \| \sigma(x) \cdot y \|_V \le | \sigma
  |_{\Cb^0(U_1,\LB(U_2,V))} \| y \|_{U_2}$ for all $(x,y) \in U_1 \times U_2$.
  Further, it holds
  \begin{align*}
    \gamma'(x,y) \cdot (z_1,z_2) = \sigma'(x) \cdot( z_1, y) + \sigma(x)
    \cdot z_2,
  \end{align*}
  for all $(x,y),(z_1, z_2) \in U_1 \times U_2$. Therefore,
  \begin{align*}
    \begin{split}
      \big\| \gamma'(x,y) \cdot (z_1,z_2) \big\|_V
      &\le | \sigma |_{\Cb^1(U_1,\LB(U_2,V))} \| z_1 \|_{U_1} \|y \|_{U_2} + |
      \sigma |_{\Cb^0(U_1,\LB(U_2,V))} \| z_2\|_{U_2}\\
      &\le \max\{| \sigma |_{\Cb^0(U_1,\LB(U_2,V))}, | \sigma
      |_{\Cb^1(U_1,\LB(U_2,V))} \} \\
      &\qquad \times \big( 1 +
      \|y\|_{U_2} \big) \big( \|z_1 \|_{U_1} + \| z_2\|_{U_2} \big).
    \end{split}
  \end{align*}
  Hence, $\gamma$ satisfies the assumption of
  Lemma~\ref{lemma3:chainrule} with $r = 1$. Thus, the result follows
  from an application of Lemma~\ref{lemma3:chainrule}.
\end{proof}

\subsection{Duality}
\label{subsec3:2}
For any $2\leq p<\infty$, $2\leq q\leq\infty$ the inclusion
$\mathbf{M}^{1,p,q}(H)\subset L^2(\Omega,H)$ is dense and continuous and hence
the spaces
\begin{align*}
  \mathbf{M}^{1,p,q}(H)\subset L^2(\Omega,H)\subset \mathbf{M}^{1,p,q}(H)^*,
\end{align*}
define a Gelfand triple, where we identify $L^2(\Omega,H)\cong L^2(\Omega,H)^*$
by the Riesz Representation Theorem. We denote the dual pairing of
$\mathbf{M}^{1,p,q}(H)^*$ and $\mathbf{M}^{1,p,q}(H)$ by $[Z,{Y}]$ for
$Z\in\mathbf{M}^{1,p,q}(H)^*$, ${Y}\in\mathbf{M}^{1,p,q}(H)$. The inclusion
$L^2(\Omega,H)\subset\mathbf{M}^{1,p,q}(H)^*$ is realized through the definition
$[Z,{Y}]=\langle Z,{Y}\rangle_{L^2(\Omega,H)}$ for all $Z\in L^2(\Omega,H)$,
${Y}\in\mathbf{M}^{1,p,q}(H)$, with the norm
\begin{align}
  \label{eq2:dual}
  \|Z\|_{\mathbf{M}^{1,p,q}(H)^*} =\sup_{Y\in
  \mathbf{M}^{1,p,q}(H)}\frac{\langle{Y},
  Z\rangle_{L^2(\Omega,H)}}{\|Y\|_{\mathbf{M}^{1,p,q}(H)}},\quad Z\in
  L^2(\Omega,H).
\end{align}

The Burkholder type inequality in Lemma \ref{lemma2:Burkholder} gives
an estimate of the norm of a stochastic integral that is $L^2$ in
time. We will now prove a similar inequality with respect to the
$\mathbf{M}^{1,p,q}(H)^*$-norm, which is $L^{q'}$ in time, where $q'$
is the conjugate exponent to $q$ given by $\frac{1}{q} + \frac{1}{q'}
= 1$ if $q<\infty$ and $q'=1$ otherwise. Since $q\in[2,\infty]$, and
hence $q'\in[1,2]$, this admits worse singularities than in
Lemma~\ref{lemma2:Burkholder}.
\begin{theorem}
  \label{lemma2:dual3}
  Let $p\in[2,\infty)$, $q\in[2,\infty]$ and $p'$, $q'$ denote the
  conjugate exponents.  If $\Phi\in L^2([0,T]\times\Omega,\LB_2^0)$ is
  predictable, then
  \begin{align*}
    \Big\|\int_0^T\Phi(t)\diff{W(t)}\Big\|_{\mathbf{M}^{1,p,q}(H)^*}\leq
    \|\Phi\|_{L^{p'}(\Omega,L^{q'}([0,T],\LB_2^0))}.
  \end{align*}
\end{theorem}

\begin{proof}
  We use the fact that the stochastic integral of $\Phi$ equals $\delta\Phi$.
  By \eqref{eq2:dual}, \eqref{eq2:IBP}, and H\"{o}lder's inequality, we get
  \begin{align*}
    \big\|\delta\Phi\big\|_{\mathbf{M}^{1,p,q}(H)^*}
    &= \sup_{Y\in \mathbf{M}^{1,p,q}(H)}\frac{\big\langle Y,
    \delta\Phi\big\rangle_{L^2(\Omega,H)}}{ \|Y\|_{\mathbf{M}^{1,p,q}(H)}} =
    \sup_{Y\in \mathbf{M}^{1,p,q}(H)}\frac{\big\langle DY, \Phi \big
    \rangle_{L^2([0,T]\times\Omega,\LB_2^0)}}{ \|Y\|_{\mathbf{M}^{1,p,q}(H)}}\\
    &\leq\sup_{Y\in \mathbf{M}^{1,p,q}(H)}
    \frac{\|DY\|_{L^p(\Omega,L^q([0,T],\LB_2^0))}
    \big\|\Phi\big\|_{L^{p'}(\Omega,L^{q'}([0,T],\LB_2^0))}}
    {\|Y\|_{\mathbf{M}^{1,p,q}(H)}}\\
    &\leq\|\Phi\|_{L^{p'}(\Omega,L^{q'}([0,T],\LB_2^0))},
  \end{align*}
  which finishes the proof.
\end{proof}

\begin{remark} Since the inequality in Lemma~\ref{lemma2:Burkholder}
  is actually double-sided, one may ask whether this is true also for
  Theorem~\ref{lemma2:dual3}.  In fact we can prove the reverse
  inequality for deterministic $\Phi\in L^2([0,T],\LB_2^0)$. Since
  $\mathcal{H}_1^{q}(H):=\{\delta\Psi : \Psi\in
  L^q([0,T],\LB_2^0)\}\subset\mathbf{M}^{1,p,q}(H)$ we get an
  inequality in \eqref{eq2:dual} by taking the supremum over
  $\mathcal{H}_1^{q}(H)$ instead of $\mathbf{M}^{1,p,q}(H)$:
  \begin{align*}
    \big\|\delta\Phi\big\|_{\mathbf{M}^{1,p,q}(H)^*}
    &=\sup_{Y\in\mathbf{M}^{1,p,q}(H)}\frac{\big\langle Y, \delta\Phi
    \big \rangle_{L^2([0,T]\times\Omega,\LB_2^0)}}{
    \|Y\|_{\mathbf{M}^{1,p,q}(H)}}\\
    &\geq \sup_{Y\in\mathcal{H}_1^{q}(H)}\frac{\big\langle DY, \Phi
    \big \rangle_{L^2([0,T]\times\Omega,\LB_2^0)}}{
    \|Y\|_{\mathbf{M}^{1,p,q}(H)}}\\
    &=\sup_{\Psi\in L^q([0,T],\LB_2^0)}\frac{\big\langle D\delta\Psi, \Phi \big
    \rangle_{L^2([0,T]\times\Omega,\LB_2^0)}}
    {\Big(\|\delta\Psi\|_{L^p(\Omega,H)}^p
    +\|D\delta\Psi\|_{L^p(\Omega,L^q([0,T],\LB_2^0))}^p\Big)^\frac1p}.
  \end{align*}
  We next use the fact that $D\delta\Psi=\Psi+\delta D\Psi=\Psi$ for
  deterministic $\Psi\in L^q([0,T],\LB_2^0)$. By Burkholder's
  inequality Lemma~\ref{lemma2:Burkholder} and H\"{o}lder's inequality
  we get
  \begin{align*}
    \big\|\delta\Phi\big\|_{\mathbf{M}^{1,p,q}(H)^*}
    &\geq\sup_{\Psi\in L^q([0,T],\LB_2^0)}\frac{\big\langle \Psi, \Phi \big
    \rangle_{L^2([0,T],\LB_2^0)}} {\Big(C_p^p\|\Psi\|_{L^2([0,T],\LB_2^0)}^p
    +\|\Psi\|_{L^q([0,T],\LB_2^0)}^p\Big)^\frac1p}\\
    &\geq\frac1{\big(C_p^p T^\frac{q}{q-2}+1\big)^\frac1p}\sup_{\Psi\in
    L^q([0,T],\LB_2^0)}\frac{\big\langle \Psi, \Phi \big
    \rangle_{L^2([0,T],\LB_2^0)}} {\|\Psi\|_{L^q([0,T],\LB_2^0)}}\\
    &=\frac1{\big(C_p^p
    T^\frac{q}{q-2}+1\big)^\frac1p}\|\Phi\|_{L^{q'}([0,T],\LB_2^0)}.
\end{align*}
The proof relies on the fact that $D\Psi=0$. For random $\Phi$ one
needs random $\Psi\in L^p(\Omega,L^q([0,T],\LB_2^0))$ and, since
$\delta D\Psi\neq0$ in this case, this proof does not work.
\end{remark}

\begin{remark}
  One consequence of Theorem~\ref{lemma2:dual3} is that the stochastic
  integral can be extended in $\mathbf{M}^{1,p,q}(H)^*$ to integrands
  in $L^{p'}(\Omega,L^{q'}([0,T],\LB_2^0))$. The elements of
  $\mathbf{M}^{1,p,q}(H)^*$ are distributions defined by their action
  on random variables in $\mathbf{M}^{1,p,q}(H)$. One can show that
  the solution of the linear stochastic heat equation driven by
  space-time white noise in two space dimensions is a stochastic
  process $X\in \Cc([0,T],\mathbf{M}^{1,p,q}(H)^*)$ for every $p\geq2$
  and $q>2$. In three space dimensions the same is valid for every
  $p\geq2$ and $q>4$. In higher space dimensions than three the
  solution is not $\mathbf{M}^{1,p,q}(H)^*$-valued since this would
  force $q'<1$. H\"{older} continuity in time in the
  $\mathbf{M}^{1,p,q}(H)^*$-norms can be shown for the solution in two
  and three space dimensions for the $p,q$ for which the solution is
  defined. See Lemma~\ref{lemma5:Lipschitz} below for the regular
  case.  Solutions defined in a distributional sense with respect to
  $\Omega$ is not a new concept. This is the heart of the white noise
  approach to SPDE, see, e.g., \cite{benth1997}, \cite{oksendal2009}.
\end{remark}

Theorem~\ref{lemma2:dual3} is a key result in the present work. But to
be able to perform error estimates for semilinear equations we also
need an intermediate space between $\mathbf{M}^{1,p,p}(H)$ and
$\mathbf{M}^{1,2p,p}(H)$. For $2\leq p <\infty$ we define
\begin{align*}
  \mathbf{G}^{1,p}(H)=\mathbf{M}^{1,p,p}(H)\cap L^{2p}(\Omega,H).
\end{align*}
It is a Banach space equipped with the norm
\begin{align*}
  \|Y\|_{\mathbf{G}^{1,p}(H)}=\max\Big(\|Y\|_{\mathbf{M}^{1,p,p}(H)},
  \|Y\|_{L^{2p}(\Omega,H)}\Big).
\end{align*}
We have $\mathbf{M}^{1,2p,p}(H) \subset \mathbf{G}^{1,p}(H) \subset
\mathbf{M}^{1,p,p}(H)$ and we obtain a new Gelfand triple
\begin{align*}
  \mathbf{G}^{1,p}(H)\subset L^2(\Omega,H)\subset \mathbf{G}^{1,p}(H)^*.
\end{align*}
The next lemma is a slightly modified version of
Lemma~\ref{lemma3:prodrule}, which is necessary to prove the local
Lipschitz bound in Lemma~\ref{lemma5:Lipschitz}.

\begin{lemma}
  \label{lemma3:prodrule2}
  Let $U_1, U_2, V$ be separable Hilbert spaces. For $\sigma \in
  \Cb^0(U_1, \LB(U_2,V)) \cap \Cb^1(U_1, \LB(U_2,V))$, $X \in
  \mathbf{M}^{1,2p,p}(U_1)$, $Y \in \mathbf{G}^{1,p}(U_2)$, $2 < p
  <\infty$, it holds $\sigma(X) \cdot Y \in \mathbf{G}^{1,p}(V)$. In
  addition, we have
  \begin{align*}
    &\| \sigma(X) \cdot Y \|_{\mathbf{G}^{1,p}(V)}\\
    &\quad \le \max \big(
    |\sigma|_{\Cb^0(U_1,\LB(U_2,V))}, |\sigma|_{\Cb^1(U_1,\LB(U_2,V))} \big)
    \Big(1 +  \| X \|_{\mathbf{M}^{1,2p,p}(U_1)} \Big) \| Y
      \|_{\mathbf{G}^{1,p}(U_2)}.
  \end{align*}
\end{lemma}

\begin{proof}
  It particularly holds $X \in \mathbf{M}^{1,p,p}(U_1)$, $Y \in
  \mathbf{M}^{1,p,p}(U_2)$, $p > 2$, and, hence,
  we directly obtain from Lemma~\ref{lemma3:prodrule} that $\sigma(X) \cdot Y
  \in \mathbf{M}^{1,\frac{p}{2},p}(V)$.
  In addition, we get
  \begin{align*}
    \| \sigma(X) \cdot Y \|_{L^{2p}(\Omega,V)} \le
    | \sigma |_{\Cb^0(U_1,\LB(U_2,V))} \| Y
    \|_{L^{2p}(\Omega,U)} \le | \sigma |_{\Cb^0(U_1,\LB(U_2,V))} \| Y
    \|_{\mathbf{G}^{1,p}(U)}.
  \end{align*}
  Further, by \eqref{Product} we have
  \begin{align*}
    &\big\|D(\sigma(X) \cdot Y )\big\|_{L^p(\Omega, L^p([0,T],\LB_2(H_0,V)))}\\
    &\quad = \big\| \sigma'(X) \cdot \big( DX, Y \big) + \sigma(X) \cdot
    D Y \big\|_{L^p(\Omega,L^p( [0,T],\LB_2(H_0,V)))}\\
    &\quad \le |\sigma|_{\Cb^1(U_1,\LB(U_2,V))} \Big( \E \Big[
    \| D X \|_{L^p([0,T], \LB_2(H_0,U_1))}^p\, \|
    Y\|^p_{U_2} \Big] \Big)^{\frac{1}{p}}\\
    &\qquad + |\sigma|_{\Cb^0(U_1,\LB(U_2,V))}
    \| D Y \|_{L^p(\Omega, L^p ([0,T],\LB_2(H_0,U_2)))}\\
    &\quad \le |\sigma|_{\Cb^1(U_1,\LB(U_2,V))}
    \| D X \|_{L^{2p}(\Omega, L^p([0,T], \LB_2(H_0,U_1)))} \| Y
    \|_{L^{2p}(\Omega,U_2)}\\
    &\qquad +|\sigma|_{\Cb^0(U_1,\LB(U_2,V))}  \| D Y
    \|_{L^p(\Omega, L^p( [0,T],\LB_2(H_0,U_2)))}\\
    &\quad \leq \max \big( |\sigma|_{\Cb^1(U_1,\LB(U_2,V))},
    \|\sigma\|_{\Cb^0(U_1,\LB(U_2,V))} \big)
    \Big(1 +  \| X \|_{\mathbf{M}^{1,2p,p}(U_1)} \Big) \| Y
      \|_{\mathbf{G}^{1,p}(U_2)}.
  \end{align*}
  These bounds show that $\sigma'(X) \cdot Y \in \mathbf{G}^{1,p}(V)$
  as well as the desired bound.
\end{proof}

In the proof of Lemma~\ref{lemma5:Lipschitz} below we require that a
particular random \stig{linear} operator is bounded
$\mathbf{G}^{1,p}(U)^*\to \mathbf{G}^{1,p}(V)^*$. The next \stig{lemma}
provides a criterion for this, in terms of the boundedness
$\mathbf{G}^{1,p}(V)\to \mathbf{G}^{1,p}(U)$ of a suitably defined
adjoint operator. It is also used in the proof of
Lemma~\ref{lemma3:Strong2} with a non-random operator.

\begin{lemma}\label{lemma3:ideal}
  Let $U,V$ be separable Hilbert spaces,
  $S\in L^\infty(\Omega,\LB(U,V))$, and $2\leq p<\infty$,
  $2\leq q \leq \infty$.  Define $S^*\in L^\infty(\Omega,\LB(V,U))$ as
  the almost sure adjoint of $S$, i.e., $S^*(\omega)=(S(\omega))^*$,
  {a.a.} $\omega\in\Omega$.  Let either $E=\mathbf{M}^{1,p,q}(U)$,
  $F=\mathbf{M}^{1,p,q}(V)$ or $E=\mathbf{G}^{1,p}(U)$,
  $F=\mathbf{G}^{1,p}(V)$. If $S^*\in \LB(F,E)$, then
  $S\in\LB(E^*,F^*)$ with
  $\|S\|_{\LB(E^*,F^*)}\leq \|S^*\|_{\LB(F,E)}$.  In particular, if 
  $S\in\LB(U,V)$ is non-random, then
  $\|S\|_{\LB(E^*,F^*)}\leq\|S\|_{\LB(U,V)}$.
\end{lemma}

\begin{proof}
We conclude that $S\colon F^*\to E^*$ is bounded by the duality calculation
\begin{align*}
  \|SY\|_{F^*}
& =
  \sup_{\|Z\|_F\leq1}
  \langle
    SY,Z
  \rangle_{L^2(\Omega;V)}
  =
  \|S^*\|_{\LB(F,E)}
  \sup_{\|Z\|_F\leq1}
  \Big\langle
    Y,
    \frac{
      S^*Z
    }{
      \|S^*\|_{\LB(F,E)}
    }
  \Big\rangle_{L^2(\Omega;U)}\\
& \leq
  \|S^*\|_{\LB(F,E)}
  \sup_{\|Z\|_E\leq1}
  \langle
    Y,Z
  \rangle_{L^2(\Omega;U)}
  =
    \|S^*\|_{\LB(F,E)}
    \|Y\|_{E^*}.
\end{align*}
Consider non-random $S\in\LB(U,V)$. For $E=\mathbf{M}^{1,p,q}(U)$, $F=\mathbf{M}^{1,p,q}(V)$ we note 
that $\|S^*\|_{\LB(F,E)}\leq\|S^*\|_{\LB(V,U)}=\|S\|_{\LB(U,V)}$,
because $D\stig{S^*}Z=\stig{S^*}DZ$ for $Z\in\mathbf{M}^{1,p,q}(\stig{V})$. The case 
$E=\mathbf{G}^{1,p}(U)$, $F=\mathbf{G}^{1,p}(V)$ follows from this.
\end{proof}

Our next key result is stated in Lemma~\ref{lemma5:Lipschitz}
below. It establishes a local Lipschitz bound in the
$\mathbf{G}^{1,p}(H)^*$-norm.  This allows us to perform a Gronwall
argument in this norm in Section~\ref{subsec4:2}.

\begin{lemma}
  \label{lemma5:Lipschitz}
  Let $U,V$ be separable Hilbert spaces, $\eta \in \Cb^2(U,V)$, and
  $2<p<\infty$.  Then, for all $X_1, X_2 \in \mathbf{M}^{1,2p,p}(U)$,
  \begin{align*}
    &\big\| \eta(X_1) - \eta(X_2) \big\|_{\mathbf{G}^{1,p}(V)^\ast}\\
    &\quad \le \max \big(
    |\eta|_{\Cb^1(U,V)},|\eta|_{\Cb^2(U,V)}
    \big) \Big(1 +  \sum_{i =1}^2 \|X_i \|_{\mathbf{M}^{1,2p,p}(U)} \Big)
    \big\| X_1 - X_2 \big\|_{\mathbf{G}^{1,p}(U)^\ast}.
  \end{align*}
\end{lemma}

\begin{proof}
  In view of \eqref{eq2:Taylor2} it suffices to show
  \begin{align*}
    \begin{split}
      &\|\eta'(X)\cdot Y\|_{\mathbf{G}^{1,p}(V)^\ast}\\
      &\quad \leq \max \big(
      |\eta|_{\Cb^1(U,V)},|\eta|_{\Cb^2(U,V)}
      \big) \Big(1 +  \| X \|_{\mathbf{M}^{1,2p,p}(U)} \Big)  \| Y
      \|_{\mathbf{G}^{1,p}(U)^\ast},
    \end{split}
  \end{align*}
  for all $X,Y\in\mathbf{M}^{1,2p,p}(U)$. By Lemma~\ref{lemma3:ideal} we have
  \begin{align*}
    \| \eta'(X) \cdot Y \big\|_{\mathbf{G}^{1,p}(V)^\ast} \leq \big\|
    \eta'(X)^* \big\|_{\LB(\mathbf{G}^{1,p}(V),\mathbf{G}^{1,p}(U))}
    \big\| Y \big\|_{\mathbf{G}^{1,p}(U)^*}.
  \end{align*}
  To bound
  $
  \|
    \eta'(X)^* 
  \|_{\LB(\mathbf{G}^{1,p}(V),\mathbf{G}^{1,p}(U))}
  $
  we define
  $\sigma \colon U \to \LB(V,U)$ by
  \begin{align*}
    \sigma(x) := \eta'(x)^\ast.
  \end{align*}
  Then $\sigma \in \Cb^0(U, \LB(V,U)) \cap \Cb^1(U, \LB(V,U))$
  with $| \sigma |_{\Cb^0(U,\LB(V,U))} = | \eta |_{\Cb^1(U,V)}$ and $|
  \sigma |_{\Cb^1(U,\LB(V,U))} = | \eta |_{\Cb^2(U,V)}$. Hence, the
  assertion follows directly from an application of
  Lemma~\ref{lemma3:prodrule2}.
\end{proof}

\subsection{Regularity of the solution}
\label{subsec3:3}
Here we prove regularity in terms of the Malliavin derivative, as well
as H\"{o}lder continuity in the $\mathbf{M}^{1,p,q}(H)^*$-norm, of the
solution $X$ to \eqref{eq1:SPDEmild} under Assumption~\ref{as1:A}. For
suitably chosen $p$ and $q$ the H\"{o}lder exponent turns out to be
twice as high as in the $L^2(\Omega,H)$-norm. By combining these
results with a duality argument we show H\"{o}lder continuity of the
Markov semigroup. The H\"older exponent is later, in
Theorem~\ref{thm1:main}, shown to coincide with the rate of weak
convergence, which is natural.

The Malliavin derivative $D_rX(t)$ of $X(t)$ at time $r\in[0,T]$
satisfies the equation, see \cite{FuhrmanTessitore}*{Proposition
  3.5~(ii)},
\begin{equation}
  \label{eq3:DX}
  D_r X(t) =
  \begin{cases}
  S(t-r)+ \displaystyle{\int_r^t S(t-s) F'(X(s)) D_r X(s)\diff{s}}, &\quad
  t\in (r,T],\\
0, &\quad t\in[0,r].
  \end{cases}
\end{equation}
The next result can be verified by using (3.11) of
\cite{FuhrmanTessitore}*{Proposition 3.5 (ii)} and holds for
multiplicative noise, as well. For completeness we present a proof in
the simpler case of additive noise that we consider here.

\begin{proposition}
  \label{prop3:reg}
  Let Assumption~\ref{as1:A} hold and let $X$ be the solution of
  \eqref{eq1:SPDEmild}. If $\beta\in(0,1)$, then
  \begin{align*}
    \sup_{t\in[0,T]}\big\|X(t)\big\|_{\mathbf{M}^{1,p,q}(H)}<\infty,
  \end{align*}
  for $2\leq p<\infty$ and $2\leq q<\tfrac2{1-\beta}$. If $\beta=1$, then
  the same holds for $2\leq p<\infty$ and $2\leq q\leq\infty$.
\end{proposition}

\begin{proof}
  We remark that the case $p=q=2$ was already proved in
  \cite{FuhrmanTessitore}. The moment estimate \eqref{ineq2:moment}
  implies that $\sup_{t\in[0,T]}\|X(t)\|_{L^{p}(\Omega,H)}<\infty$ for
  $2\leq p<\infty$.  Next we take norms in \eqref{eq3:DX} and use
  Minkowski's inequality on the convolution term.  We note that
  $D_rX(s)=0$ for $s\le r$ because $X(s)$ is $\F_r$-measurable, so
  that the convolution term can be written $\int_0^t \dots \diff{s}$.
  We get
  \begin{equation*}
    \begin{split}
&\big\|DX(t)\big\|_{L^p(\Omega,L^q([0,T],\LB_2^0))}
=
\big\|DX(t)\big\|_{L^p(\Omega,L^q([0,t],\LB_2^0))}
\\ & \quad
\leq
\big\|S(t-\cdot)\big\|_{L^q([0,t],\LB_2^0)}
+
\Big\|
\int_{0}^t
S(t-s) F'(X(s)) D X(s)\diff{s}
\Big\|_{L^p(\Omega,L^q([0,t],\LB_2^0))}
\\
& \quad \leq
\big\| S \big\|_{L^q([0,t],\LB_2^0)}
+\int_0^t
\big\| S(t-s) F'(X(s)) DX(s)
\big\|_{L^p(\Omega,L^q([0,t],\LB_2^0))} \diff{s}
\\ & \quad\leq
\big\|S\big\|_{L^q([0,T],\LB_2^0)}
+\|S\|_{L^\infty([0,T],\LB)}|F|_{\Cb^1}
\int_0^t \big\|DX(s)\big\|_{L^p(\Omega,L^q([0,T],\LB_2^0))}\diff{s}.
    \end{split}
  \end{equation*}
  We conclude by using Lemma~\ref{lemma:Slqbound} and the standard
  Gronwall lemma.
\end{proof}

We next consider H\"older continutity in the
$\mathbf{M}^{1,p,q}(H)^\ast$-norm.  For comparison we recall that the
H\"older exponent in the $L^2(\Omega,H)$-norm is $\gamma<\beta/2$
under Assumption~\ref{as1:A}. Here we have $\gamma<\beta$, if $q$ is
sufficiently large.

\begin{proposition}
  \label{lemma4:tempreg}
  Let Assumption~\ref{as1:A} hold with $\beta \in (0,1]$ and denote by
  $X$ the solution to \eqref{eq1:SPDEmild}. Let $2\le p<\infty$,
  $\gamma\in[0,\beta)$, and set $q=\tfrac2{1-\gamma}$. Then there
  exists a constant $C=C_\gamma$ such that
  \begin{align*}
    \big\| X(t_2)-X(t_1) \big\|_{\mathbf{M}^{1,p,q}(H)^\ast} \le\,
    C \Big(1+\big\|X_0\big\|_{\dot{H}^{2\beta}}\Big)\big| t_2 -
    t_1\big|^{\gamma},\quad t_1, t_2 \in [0,T].
  \end{align*}
\end{proposition}

\begin{proof}
  Without loss of generality we assume $t_2 > t_1>0$. From
  \eqref{eq1:SPDEmild} we then get
  \begin{align*}
    X(t_2) - X(t_1) &= \big( S(t_2 - t_1) - I \big) S(t_1) X_0\\
    &\quad + \big( S(t_2 - t_1) - I \big) \int_{0}^{t_1}
    S(t_1 - s) F(X(s)) \diff{s}\\
    &\quad + \big( S(t_2 - t_1) - I \big) \int_{0}^{t_1} S(t_1 - s)
    \diff{W(s)}\\
    &\quad + \int_{t_1}^{t_2} S(t_2 - s) F(X(s)) \diff{s} + \int_{t_1}^{t_2}
    S(t_2 - s) \diff{W(s)}.
  \end{align*}
  In the following we study the $\mathbf{M}^{1,p,q}(H)^\ast$-norms of
  these five summands. For the first, second, and fourth terms we use
  the fact that $\| Z \|_{\mathbf{M}^{1,p,q}(H)^\ast} \le \| Z
  \|_{L^2(\Omega,H)}$.

  For the first summand, we use \eqref{as1:SAnalytic2} with
    $\varrho = \gamma$ and \eqref{as1:SAnalytic} with $\varrho=0$ as
    well as Assumption~\ref{as1:A} (ii). This yields
  \begin{align*}
    \big\| \big( S(t_2 - t_1) - I \big)S(t_1) X_0
    \big\|_{\mathbf{M}^{1,p,q}(H)^\ast}
    &\le \big\| \big( S(t_2 - t_1) - I \big) A^{-\gamma} S(t_1)A^\gamma  X_0
    \big\|_{L^2(\Omega,H)} \\
    &\lesssim \big|t_2 - t_1\big|^{\gamma} \|A^\gamma X_0 \|
  \lesssim \big|t_2 - t_1\big|^{\gamma} \|X_0 \|_{\dot{H}^{2\beta}}.
  \end{align*}

  The estimate of the second summand is done by applying Assumption~\ref{as1:A}
  (iv) and the same arguments as for the first term. More precisely,
  we use that $F\in \Cb^1(H,H)$ implies linear growth, to get
  \begin{align*}
    &\Big\| \big( S(t_2 - t_1) - I \big) \int_{0}^{t_1}
    S(t_1 - s) F(X(s)) \diff{s} \Big\|_{\mathbf{M}^{1,p,q}(H)^\ast} \\
    &\quad \le \big\| \big(S(t_2 - t_1) - I \big) A^{-\gamma}
    \big\|_{\LB} \int_{0}^{t_1} \big\| A^{\gamma} S(t_1 - s) \big\|_{\LB}
    \big\| F(X(s)) \big\|_{L^2(\Omega,H)} \diff{s}\\
    &\quad \lesssim \big|t_2 - t_1\big|^{\gamma} \int_{0}^{t_1} (t_1 -
    s)^{-\gamma} \diff{s}\, \Big( 1+ \sup_{s \in [0,T]} \| X(s)
    \|_{L^2(\Omega,H)} \Big)\,\lesssim\, \big|t_2 - t_1\big|^{\gamma},
  \end{align*}
  where we also used \eqref{ineq2:moment} and that $\gamma<\beta\le1$.

  We now turn to the third term. We recall that
  $q=2/(1-\gamma)$ and $q'=2/(1+\gamma)$. Since $\gamma<\beta$, we
  have
  \begin{align}
    \label{calc3:exponent}
    q'\frac{2\gamma+1-\beta}2 =\frac{2\gamma+1-\beta}{1+\gamma}
    = 1-\frac{\beta-\gamma}{1+\gamma} < 1.
  \end{align}
  We apply Theorem~\ref{lemma2:dual3} to the third summand.  Then by
  \eqref{as1:SAnalytic}, \eqref{as1:SAnalytic2},
  Assumption~\ref{as1:A} (iii), and \eqref{calc3:exponent}, we obtain
  \begin{align*}
    & \Big\| \big( S(t_2 - t_1) - I \big) \int_{0}^{t_1} S(t_1 - s)
    \diff{W(s)} \Big\|_{\mathbf{M}^{1,p,q}(H)^\ast}\\
    &\quad \le \big\| \big( S(t_2 - t_1) - I \big) S(t_1 - \cdot)
    \big\|_{L^{p'}(\Omega,L^{q'}([0,t_1], \LB_2^0))}\\
    &\quad \le \big\| \big(S(t_2 - t_1) - I \big) A^{-\gamma}
    \big\|_{\LB} \Big( \int_{0}^{t_1} \big\| A^{\gamma} A^{\frac{1-\beta}{2}}S(t_1 - s)A^{\frac{\beta-1}{2}}
    \big\|^{q'}_{\LB_2^0} \diff{s} \Big)^{\frac{1}{q'}} \\
    &\quad \lesssim \big|t_2 - t_1\big|^{\gamma}
    \Big( \int_{0}^{t_1} (t_1 - s)^{-q'\frac{2\gamma + 1 - \beta }{2}}
    \diff{s} \, \big\| A^{\frac{\beta -1}{2}} \big\|_{\LB_2^0}^{q'} \Big)^{\frac{1}{q'}}
    \,\lesssim\, \big|t_2 - t_1\big|^{\gamma}.
  \end{align*}

  Next we turn to the fourth term. By applying the same arguments as for the
  second summand, we derive the bound
  \begin{align*}
     \Big\| \int_{t_1}^{t_2} S(t_2 - s) F(X(s)) \diff{s}
    \Big\|_{\mathbf{M}^{1,p,q}(H)^\ast} & \le \int_{t_1}^{t_2} \big\| S(t_2 - s)
    F(X(s)) \big\|_{L^2(\Omega,H)} \diff{s} \\
    & \lesssim | t_2 - t_1 | \Big( 1 + \sup_{s \in [0,T]} \big\| X(s)
    \big\|_{L^2(\Omega,H)} \Big).
  \end{align*}

  Finally, a further application of Theorem~\ref{lemma2:dual3} and
  \eqref{as1:SAnalytic} with $\varrho = \tfrac{1-\beta}2$ yields for the
  fifth summand
\begin{align*}
    \Big\| \int_{t_1}^{t_2} S(t_2 - s) \diff{W(s)}
    \Big\|_{\mathbf{M}^{1,p,q}(H)^\ast}
& \le \Big(
\int_{t_1}^{t_2} \big\| S(t_2-s)
A^{\frac{1-\beta}{2}}\big\|_{\LB}^{q'}
\big\| A^{\frac{\beta - 1}{2}}\big\|^{q'}_{\LB_2^0}
\diff{s} \Big)^{\frac{1}{q'}} \\
&
\lesssim \Big(
\int_{t_1}^{t_2}  (t_2-s)^{- \frac{1-\beta}{2}}
\diff{s} \Big)^{\frac{1}{q'}}
\lesssim |t_2 - t_1|^{\frac{1}{q'}- \frac{1-\beta}{2}}.
  \end{align*}
  By inserting $q' = 2/ (1+ \gamma)$ and $\beta>\gamma$, we see that
  the exponent is
  \begin{align*}
\frac{1}{q'}- \frac{1-\beta}{2} = \frac{1+\gamma}{2}-
\frac{1-\beta}{2}= \frac{\gamma+\beta}{2}>\gamma.
  \end{align*}
This completes the proof.
\end{proof}

As a consequence of Propositions~\ref{prop3:reg} and
\ref{lemma4:tempreg} we now show H\"{o}lder continuity of the
Markov semigroup $(P(t))_{t\in[0,T]}$ related to $X$. This will not be
used in the sequel but it is a neat application of the duality
argument. A similar result, which we are aware of, is 
\cite {jentzen2010}*{Corollary 7}. 
Define for $(t,x)\in[0,T]\times H$,
$(P(t)\varphi)(x)=\E[\varphi(X(t,x))]$, where $X(t,x)$ denotes the
solution to equation \eqref{eq1:SPDEmild} with initial value $X_0=x\in
\dot{H}^{2\beta}$.
\begin{corollary}
  \label{cor:Markov}
  Let Assumption~\ref{as1:A} hold with $\beta\in(0,1]$ and let
  $\varphi\in \Cp^2(H,\R)$. For every
  $\gamma\in[0,\beta)$ there is a constant $C$ such that
  \begin{align*}
    \big|(P(t_2)\varphi)(x)-(P(t_1)\varphi)(x)\big|\leq
    C\big(1+\|x\|_{\dot{H}^{2\beta}}\big)\big| t_2-t_1\big|^\gamma ,\quad
    t_1,t_2\in[0,T],\ x\in \dot{H}^{2\beta}.
  \end{align*}
\end{corollary}

\begin{proof}
  We fix $x$ and suppress it from the notation. Applying \eqref{eq2:Taylor2}
  yields
  \begin{align*}
    &\big| (P(t_2)\varphi)(x)-(P(t_1)\varphi)(x)\big|
    =\big|\E\big[\varphi(X(t_2))-\varphi(X(t_1))\big]\big|\\
    &\quad=\Big|\Big\langle\int_0^1\varphi'\big(\varrho
    X(t_2)+(1-\varrho)X(t_1)\big)\diff{\varrho},X(t_2)-X(t_1)
    \Big\rangle_{L^2(\Omega,H)}\Big|.
  \end{align*}
  For arbitrary $p \in [2,\infty)$ we obtain by duality
  \begin{align*}
    &\big| (P(t_2)\varphi)(x)- (P(t_1)\varphi)(x)\big|\\
    &\quad \leq \Big\|\int_0^1\varphi'\big(\varrho
    X(t_2)+(1-\varrho)X(t_1)\big)\diff{\varrho}\Big\|_{\mathbf{M}^{1,p,p}(H)}
    \big\|X(t_2)-X(t_1)\big\|_{\mathbf{M}^{1,p,p}(H)^*}.
  \end{align*}
  Now take $p=\frac{2}{1-\gamma}$.  The first factor is finite by
  Proposition~\ref{prop3:reg} and the chain rule; for details see the
  proof of Lemma~\ref{lemma4:Phi}
  below. Proposition~\ref{lemma4:tempreg} applies to the second factor
  and this completes the proof.
\end{proof}

\begin{remark}
  Proposition~\ref{lemma4:tempreg} can be proved without additional
  difficulties in the case of multiplicative noise and so can
  Proposition~\ref{prop3:reg}, due to the comment right before its
  statement. Therefore, Corollary~\ref{cor:Markov} holds for
  multiplicative noise.
\end{remark}

\begin{remark}
  We end this section with a comment on implications to stochastic
  ordinary differential equations. This corresponds to the case $A=0$,
  $\beta=1$, and multiplicative noise with diffusion coefficient $G\in
  \Cb^2(H,\LB_2^0)$, i.e., we consider the equation
\begin{align}
\label{eq:SDE}
\diff{X(t)}=F(X(t))\diff{t}+G(X(t))\diff{W(t)},\ t\in(0,T];\quad X(0)=X_0.
\end{align}
In this case one can prove Proposition~\ref{lemma4:tempreg} with
$p\geq2$, $q=\infty$, and $\gamma=1$, meaning that the solution is
Lipschitz continuous in time in the
$\mathbf{M}^{1,p,\infty}(H)^*$-norm for every $p\geq2$. For $\beta=1$
the covariance operator $Q$ is of trace class and the cylindrical
Wiener process $W$ is well defined as an $H$-valued Brownian
motion. We see that also $W$ is Lipschitz continuous in
$\mathbf{M}^{1,p,\infty}(H)^*$ by Proposition
\ref{lemma2:dual3}. Indeed,
\begin{align*}
&\big\|W(t_2)-W(t_1)\big\|_{\mathbf{M}^{1,p,\infty}(H)^*}
= \Big\|\int_{t_1}^{t_2}\diff{W(t)}\Big\|_{\mathbf{M}^{1,p,\infty}(H)^*}\\
&\qquad\leq
\big\|\chi_{[t_1,t_2]}\big\|_{L^{p'}(\Omega,L^1([0,T],\LB_2^0))}
=\Tr(Q)\big|t_2-t_1\big|,\quad t_1,t_2\in[0,T].
\end{align*}
This suggests that $\diff{X(t)}=\dot{X}(t)\diff{t}$ and
$\diff{W(t)}=\dot{W}(t)\diff{t}$, where $\dot{X}$ and $\dot{W}$ are
$\mathbf{M}^{1,p,\infty}(H)^*$-valued functions on $[0,T]$. This further
suggests that \eqref{eq:SDE} might be written in the form
\begin{align*}
\dot{X}(t)=F(X(t))+G(X(t))\dot{W}(t).
\end{align*}
If this formulation is useful or fully makes sense is an open
question. There seems to be a connection to the functional white noise
approach of stochastic differential equations, see
\cite{oksendal2007}, that remains to be understood. In this approach the
time derivative of Brownian motion is well defined in the space of
Hida distributions and the corresponding product of $G$ and $\dot{W}$
is the Wick product.
\end{remark}

\subsection{Regularity of the numerical solution}
\label{subsec3:4}
Here we first show a bound on the $p$:th-moment of the discrete
solutions $X_{h,k}$ to \eqref{eq1:SPDEiterat}, uniformly in $h,k \in
(0,1]$, and then we prove a discrete analog of Proposition~\ref{prop3:reg}.

\begin{proposition}
  \label{lemma3:numstab}
  Let Assumptions~\ref{as1:A} and~\ref{as1:Scheme} hold with
  $\beta\in(0,1]$ and let $2\le p<\infty$.  Then
  \begin{align*}
    \max_{n \in \{0,\ldots,N\}} \sup_{h,k \in (0,1]} \big\| X_{h,k}^n
    \big\|_{L^p(\Omega,H)} \le C.
  \end{align*}
\end{proposition}

\begin{proof}
  For $n \in \{1,\ldots,N\}$ we recall the
  representation \eqref{eq3:SPDEfulldisc} of $X_{h,k}^n$. Hence, it
  follows that
  \begin{align*}
    \big\| X_{h,k}^n \big\|_{L^p(\Omega,H)}
    & \le \big\| S_{h,k}^n P_h X_0 \big\|
     + k \sum_{j = 0}^{n-1}
    \big\| S_{h,k}^{n-j} F(X_{h,k}^j) \big\|_{L^p(\Omega,H)}\\
    &\quad + \Big\| \int_{0}^{T} \Big( \sum_{j = 0}^{n-1}
    \chi_{[t_j,t_{j+1})}(t) S_{h,k}^{n-j} \Big) \diff{W(t)}
    \Big\|_{L^p(\Omega,H)}.
  \end{align*}
  By \eqref{as1:Analyt} with $\varrho=0$ we have
  \begin{align}
    \label{eq3:STildeBdd}
    \sup_{n \in \{1,\ldots,N\}} \big\| S_{{h},k}^n \big\|_{\LB} \lesssim 1,
  \end{align}
  so that $\| S_{h,k}^n P_h X_0 \|\lesssim 1$.  Therefore, by applying
  also Lemma~\ref{lemma2:Burkholder},
  \begin{align*}
    \big\| X_{h,k}^n \big\|_{L^p(\Omega,H)} \lesssim
    1 + k\sum_{j = 0}^{n-1}
    \big\| F(X_{h,k}^j) \big\|_{L^p(\Omega,H)}
    + \Big\|\sum_{j = 0}^{n-1} \chi_{[t_j,t_{j+1})}
    S_{h,k}^{n-j} \Big\|_{L^2([0,T],\LB_2^0)}.
  \end{align*}
  By referring to Lemma~\ref{lemma:Slqbound} with $q=2$, we have
  \begin{align*}
    &\Big\|\sum_{j = 0}^{n-1} \chi_{[t_j,t_{j+1})} S_{h,k}^{n-j}
    \Big\|_{L^2([0,T],\LB_2^0)}^2
    = k \sum_{j = 0}^{n-1}
    \big\| S_{h,k}^{n-j} \big\|_{\LB_2^0}^2
    \le k \sum_{j = 1}^{N}
    \big\| S_{h,k}^{j} \big\|_{\LB_2^0}^2
    \lesssim 1.
  \end{align*}
  Further, since the drift $F \colon H \to H$ satisfies a linear growth bound
  under Assumption~\ref{as1:A} (iv), it follows that
  \begin{align*}
    \big\| X_{h,k}^n \big\|_{L^p(\Omega,H)} \lesssim 1
     + k \sum_{j = 0}^{n-1} \big\| X_{h,k}^j \big\|_{L^p(\Omega,H)}
  \end{align*}
  and the proof is completed by an application of Gronwall's
  Lemma~\ref{lemma2:Gronwall}.
\end{proof}

\begin{proposition}
  \label{lemma3:MalliavinXhk}
  Let Assumptions~\ref{as1:A} and \ref{as1:Scheme} hold with $\beta\in(0,1]$.
  If $\beta\in(0,1)$, then
  \begin{align*}
  \max_{n\in\{1,\dots,N\}}\;\sup_{h,k\in(0,1]}
  \big\|X_{h,k}^n\big\|_{\mathbf{M}^{1,p,q}(H)}<\infty,
  \end{align*}
  for $2\leq p <\infty$ and $2\leq q <\tfrac2{1-\beta}$. If $\beta=1$,
  then the same holds for $2\leq p<\infty$ and $2\leq q\leq\infty$.
\end{proposition}

\begin{proof}
  We mimic the proof of Proposition~\ref{prop3:reg}. The
  $L^p(\Omega,H)$-norm of $X_{h,k}$ is treated in
  Proposition~\ref{lemma3:numstab} and it remains to bound $DX_{h,k}$.

  By using the chain rule \eqref{Chain1} and
  $D_r\int_{t_j}^{t_{j+1}}S_{{h},k}^{n-j}\diff{W(s)}=\chi_{[t_j,t_{j+1})}(r)
  S_{h,k}^{n-j}$, we apply the Malliavin derivative termwise to
  equation \eqref{eq3:SPDEfulldisc} and obtain
\begin{equation}  \label{eq3:DXhk}
    D_rX_{{h},k}^n
   = k\sum_{j=0}^{n-1}S_{{h},k}^{n-j}F'(X_{{h},k}^j)D_r X_{{h},k}^j
   +\sum_{j=0}^{n-1}
    \chi_{[t_j,t_{j+1})}(r) S_{h,k}^{n-j}.
\end{equation}
Here we note that $D_r X_{h,k}^{j} =0$ for $t_j\le r$, since
$X_{h,k}^j$ is $\F_r$-measurable. Therefore,
\begin{align*}
D_r X_{{h},k}^n
=
\sum_{i = 0}^{n-1}  \chi_{[t_i,t_{i+1})}(r)
\Big(
k \sum_{j = i + 1}^{n-1}
S_{h,k}^{n-j} F'( X_{h,k}^j ) D_rX_{h,k}^j
+
S_{h,k}^{n-i} \Big)
\end{align*}
in full analogy with \eqref{eq3:DX}.  However, as in the proof of
Proposition~\ref{prop3:reg}, it is more convenient to take norms in
\eqref{eq3:DXhk} and use Minkowski's inequality on the convolution
term:
\begin{align*}
&\big\|DX_{h,k}^n \big\|_{L^{p}(\Omega,L^q([0,T],\LB_2^0))}
=\big\|DX_{h,k}^n \big\|_{L^{p}(\Omega,L^q([0,t_n],\LB_2^0))}
\\ & \quad
\leq \Big\|
\sum_{j = 0}^{n-1}
\chi_{[t_j,t_{j+1})} S_{h,k}^{n-j}
\Big\|_{L^q([0,t_n],\LB_2^0)}
\\ & \qquad
+ \Big\|
k\sum_{j=0}^{n-1}S_{{h},k}^{n-j}F'(X_{{h},k}^j)D_r X_{{h},k}^j
\Big\|_{L^{p}(\Omega,L^q([0,t_n],\LB_2^0))}
\\ &  \quad
\le \Big(k\sum_{j = 1 }^{N} \| S_{h,k}^{j}\|_{\LB_2^0}^q\Big)^{1/q}
+  \sup_{1\le j\le N}
\big\| S_{{h},k}^j\big\|_{\LB}
|F|_{\Cb^1}
k \sum_{j = 0 }^{n-1}
\big\|
D X_{{h},k}^j
\big\|_{L^p(\Omega,L^q([0,T],\LB_2^0))}.
\end{align*}
We conclude by using Lemma~\ref{lemma:Slqbound},
\eqref{eq3:STildeBdd}, and  the discrete Gronwall
Lemma~\ref{lemma2:Gronwall}.
\end{proof}

\section{Weak convergence by duality}
\label{sec4}
Let $X$ be the solution to equation \eqref{eq1:SPDEmild} and $X_{h,k}$
be the discretization given by the semi-implicit scheme
\eqref{eq1:SPDEiterat} and take $\varphi\in\mathcal{G}^1(H,\R)$. Our
approach to weak convergence begins with an application of
\eqref{eq2:Taylor2} to get
\begin{align*}
  \E \big[ \varphi(X(t_n))-\varphi(X_{{h},k}^n) \big]
  &= \big\langle\Phi_{h,k}^n,  X(t_n)-X_{{h},k}^n \big\rangle_{L^2(\Omega,H)},
\end{align*}
where
\begin{align}
  \label{Theta}
  \Phi_{h,k}^n =\int_0^1 \varphi'(\Theta_{h,k}^n(\varrho))\diff{\varrho}
  \quad\textrm{and}\quad \Theta_{h,k}^n(\varrho)=\varrho
  X(t_n)+(1-\varrho)X_{h,k}^n,
\end{align}
for $n \in \{1,\ldots,N\}$. This linearization was first proposed in
\cite{KohatsuHiga2} for nonlinear stochastic ordinary differential
equations. They proceed by a duality argument based on an adjoint
equation.

This linearization was used in \cite{kruse2013} for
linear stochastic partial differential equations.  Extending the idea
of \cite{kruse2013}, we proceed as follows: choose a Gelfand triple
$V\subset L^2(\Omega,H)\subset V^*$ such that $\Phi_{h,k}^n\in V$. By
duality we have
\begin{align}
  \label{ineq:dualbound}
  \big|\E \big[ \varphi(X(t_n))-\varphi(X_{{h},k}^n) \big] \big|
  &\leq\Big(\sup_{h,k\in(0,1]}\big\|\Phi_{h,k}^n\big\|_V\Big)
  \big\|X(t_n)-X_{{h},k}^n\big\|_{V^*}.
\end{align}
The proof of our weak convergence result in Theorem~\ref{thm1:main}
then amounts to showing that we can find a suitable space $V$ such
that, for $\gamma\in(0,\beta)$,
\begin{align}
  \label{proc4:wc}
  \begin{split}
    \max_{n \in \{1,\ldots,N\}}
    \sup_{h,k\in(0,1]}\big\|\Phi_{h,k}^n\big\|_V &\le C,\\
    \max_{n \in \{1,\ldots,N\}} \big\| X(t_n)-X_{{h},k}^n \big\|_{V^*} &\leq
    C\big( h^{2\gamma}+k^\gamma\big),\quad h,k\in(0,1].
  \end{split}
\end{align}
In comparison, the strong error converges with half this rate,
i.e., for $\gamma\in(0,\beta)$ there exists $C$ such that
\begin{align*}
\max_{n \in \{1,\ldots,N\}}\|X(t_n)-X_{h,k}^n\|_{L^2(\Omega,H)}
\leq C(h^\gamma+k^\frac\gamma2),\quad h,k\in(0,1].
\end{align*}
In Corollary~\ref{cor4:strong} we deduce this from \eqref{proc4:wc} by
an interpolation argument.

We explain our method by gradually choosing more sophisticated spaces
$V$.  We begin in the next subsection with the simpler problem of the
weak approximation of the stochastic convolution. This problem is
treated in \cite{debussche2009}, \cite{Geissert} \cite{larsson2011},
\cite{larsson2013}, \cite{kruse2013}, and to some extent in
\cite{yan2005}. We show that in this case
$V=L^2(\Omega,\dot{H}^\gamma)$ and $V=\mathbf{M}^{1,p,p}(H)$ with $p =
\frac{2}{1 - \gamma}$ suffice with different degrees of success. The
proofs are simpler than in the mentioned papers, except for
\cite{kruse2013} to which the present paper is an extension. We
continue with a subsection containing our main result
Theorem~\ref{thm1:main}, which is concerned with semilinear equations
with additive noise. Here we use the space $V=\mathbf{G}^{1,p}(H)$,
whose dual norm allows for a Gronwall argument based on
Lemma~\ref{lemma5:Lipschitz}. Finally, we discuss multiplicative noise
in Subsection~\ref{subsec4:3} and illustrate why our approach is not
yet sufficient for this generality.

We assume that test functions are taken from $\Cp^2(H,\R)$ with a
precise formulation in the following assumption. Recall the norm
defined in \eqref{def2:Normb}.

\begin{assumption}
  \label{as1:phi}
  The test function $\varphi\in\Cp^2(H,\R)$ satisfies, for
  some integer $m\geq2$ and constant $C$, the bounds
  \begin{align*}
    \|\varphi^{(j)}(x)\|_{\LB^{[j]}(H,\R)}\leq C\big(1+\|x\|^{m-j}\big),\quad
    x\in H,\ j=1,2.
  \end{align*}
\end{assumption}

\subsection{The stochastic convolution}
\label{subsec4:1}

We consider the stochastic convolution $W^A$ and its approximation
$W_{h,k}^{A_h}$,
\begin{align*}
  W^A(t_n) = \int_0^{t_n} S(t_n-s) \diff{W(s)}\quad \text{and}\quad
  W_{h,k}^{A_h,n} = \sum_{j = 0}^{n - 1} \int_{t_j}^{t_{j+1}} S_{h,k}^{n-j}
  \diff{W(s)}
\end{align*}
for $n \in \{1,\ldots,N\}$. For $\gamma \in (0,\beta)$, we consider
first the Gelfand triple
\begin{align*}
  L^2\big(\Omega,\dot{H}^\gamma\big)\subset L^2(\Omega,H)\subset
  L^2\big(\Omega,\dot{H}^{-\gamma}\big).
\end{align*}
In order to have $\Phi_{h,k}^n \in{L^2(\Omega,\dot{H}^\gamma)}$ we 
impose an extra assumption on $\varphi$, namely that, for some $m\geq1$ and
every $\gamma \in (0, \beta)$, it holds
\begin{align}
  \label{as4:phi}
  \big\|\varphi'(x) \big\|_{\dot{H}^{\gamma}}\leq
  C\Big(1+\|x\|_{\dot{H}^{\gamma}}^{m-1}\Big),\quad x\in \dot{H}^{\gamma}.
\end{align}
Then, by the Sobolev regularity of $W^A$ and $W_{h,k}^{A_h}$, we get
\begin{align*}
  \big\|\Phi_{h,k}^n \big\|_{L^2(\Omega,\dot{H}^{\gamma})} &\lesssim
  \big\|W^A(t_n)\big\|_{L^{2(m-1)}(\Omega,\dot{H}^\gamma)}^{m-1}
  +\big\|W_{h,k}^{A_h,n}\big\|_{L^{2(m-1)}(\Omega,\dot{H}^\gamma)}^{m-1}
  \lesssim 1,
\end{align*}
uniformly in $h,k\in(0,1]$. To prove convergence in
$L^2(\Omega,\dot{H}^{-\gamma})$ we write the difference of the
stochastic convolution and its numerical discretization in the form
\begin{align}
\label{eq4:diffSC}
  W^A(t_n)-W_{h,k}^{A_h,n} = \int_{0}^{t_n} \tilde{E}_{h,k}(t_n-t)
  \diff{W(t)},
\end{align}
where $\tilde{E}_{h,k} \colon (0,T) \to \LB_2^0$ is given by
\begin{align}
  \label{eq4:ErrorOp}
  \tilde{E}_{h,k}(t) := S(t) - S_{h,k}^{j+1}, \quad \text{ for } t \in
  (t_{j}, t_{j+1}),\ j = 0,\ldots,N-1.
\end{align}
Under the additional assumption 
\begin{align}
  \label{eq4:negnormest}
  \big\| A^{-\frac\gamma2} \tilde{E}_{h,k}(t)
  A^\frac{1-\beta}2\big\|_{\LB}
  \lesssim \big(h^{2\gamma}+k^\gamma\big)
  t^{\frac{-1+\beta-\gamma}{2}}, \quad t>0,
\end{align}
which we only impose for this Gelfand triple, we obtain by the 
It\={o} isometry and Assumption~\ref{as1:A} (iii)
\begin{align*}
  &\big\|W^A(t_n)-W_{h,k}^{A_h,n}\big\| _{L^2(\Omega,\dot{H}^{-\gamma})}
  =\Big(\int_0^{t_n} \big\| A^{-\frac\gamma2} \tilde{E}_{h,k}(t_n-t)
  \big\|_{\LB_2^0}^2\diff{t} \Big)^\frac12\\
  &\qquad\leq \Big(\int_0^{t_n} \big\| A^{-\frac\gamma2} \tilde{E}_{h,k}(t_n-t)
  A^\frac{1-\beta}2\big\|_{\LB}^2\,
  \big\|A^{\frac{\beta-1}2}\big\|_{\LB_2^0}^2\diff{t}\Big)^\frac12\\
  &\qquad\lesssim \big(h^{2\gamma}+k^\gamma\big)
  \Big(\int_0^{t_n} (t_n-t)^{-1+\beta-\gamma} \diff{t}\Big)^\frac12\,
  \lesssim \, h^{2\gamma}+k^\gamma.
\end{align*}
Thus, in view of
\eqref{ineq:dualbound}, by assuming \eqref{as4:phi} and
\eqref{eq4:negnormest}, we can prove weak convergence with the desired
rate.

The assumption \eqref{as4:phi} is too restrictive and we
  therefore use this Gelfand triple only to demonstrate our method in
  a simple situation.  The error estimate \eqref{eq4:negnormest} is
  not to be found in the literature; except for a related error
  estimate in \cite{yan2005}, details in \cite{yan-thesis}.  As our
  main result is proved with another Gelfand triple, and without
  \eqref{eq4:negnormest}, we did not attempt to prove this.

Actually, \cite{yan2005}*{Theorem 1.2} shows convergence of order
  $O(h^{2\beta}+k^\beta)$ in $L^2(\Omega,\dot{H}^{-1})$ (except for a
  logarithmic factor). However, the fact that
$L^2(\Omega,\dot{H}^{-1})$-convergence implies weak convergence for
other than linear test functionals was not realized in the early work
\cite{yan2005}.  Subsequent works except \cite{kruse2013} rely
on the use of Kolmogorov's equation. In the paper \cite{Geissert} this
was done for test functions satisfying \eqref{as4:phi}, while
\cite{debussche2009} only assumed 
$\varphi\in\mathcal{C}_{\mathrm{b}}^2(H,\R)$. We also
remark that the only technical ingredient used in the present proof is
the It\={o} isometry. Therefore this proof carries over without
additional difficulties to the case when the cylindrical $Q$-Wiener
process $W$ is replaced by a square integrable martingale $M$, by just
introducing the suitable notation. This gives a partial extension of
the results in \cite{LindnerSchilling}, in which impulsive noise was
considered. In that paper the additional assumption \eqref{as4:phi}
was not used but instead the test functions were assumed to be in
$\mathcal{C}_{\mathrm{b}}^2(H,\R)$.

Fix $\gamma\in(0,\beta)$ and let $p=\tfrac2{1-\gamma}$. We next
consider the Gelfand triple
\begin{align*}
  \mathbf{M}^{1,p,p}(H)\subset L^2(\Omega,H)\subset \mathbf{M}^{1,p,p}(H)^*.
\end{align*}
With these spaces we need no assumption on the test function other
than Assumption~\ref{as1:phi} and we do not use
  \eqref{eq4:negnormest}. We state the two parts of \eqref{proc4:wc}
as two separate lemmas. Notice that the first lemma is not restricted
to the stochastic convolution.

\begin{lemma}
  \label{lemma4:Phi}
  Let Assumptions~\ref{as1:A}, \ref{as1:Scheme}, and \ref{as1:phi}
  hold with $\beta\in(0,1]$.  For $\gamma\in(0,\beta)$, set
  $p=\tfrac2{1-\gamma}$. Then it holds
  \begin{align*}
    \max_{n \in \{1,\ldots,N\}} \sup_{h,k \in (0,1]}
    \big\|\Phi_{h,k}^n \big\|_{\mathbf{M}^{1,p,p}(H)} < \infty,
  \end{align*}
  where $\Phi_{h,k}^n$ is defined in \eqref{Theta}.
\end{lemma}

\begin{proof}
  First note that $\varphi'$ satisfies the condition of the chain rule
  in Lemma~\ref{lemma3:chainrule} with $r = m-2$ by
  Assumption~\ref{as1:phi}. Thus, it holds
  \begin{align*}
    \Phi_{h,k}^n = \int_0^1\varphi'(\Theta_{h,k}^n(\varrho)) \diff{\varrho}
    \in \mathbf{M}^{1,p,p}(H),
  \end{align*}
  since $\Theta_{h,k}^n(\varrho) = \varrho X(t_n) +
  (1-\varrho)X_{h,k}^n \in \mathbf{M}^{1,(m-1)p,p}(H)$ by
  Propositions~\ref{prop3:reg} and \ref{lemma3:MalliavinXhk}. Further,
  from Lemma~\ref{lemma3:chainrule} we also get
  \begin{align*}
    \big\|\Phi_{h,k}^n \big\|_{\mathbf{M}^{1,p,p}(H)} &\lesssim \big(1 +
    \sup_{\varrho \in [0,1]} \big\|\Theta_{h,k}^n
    \big\|_{\mathbf{M}^{1,(m-1)p,p}(H)}^{m-1} \big)\\
    &\lesssim
    \big(1 + \big\|X(t_n) \big\|_{\mathbf{M}^{1,(m-1)p,p}(H)}^{m-1}
    + \big\|X_{h,k}^n \big\|_{\mathbf{M}^{1,(m-1)p,p}(H)}^{m-1} \big).
  \end{align*}
  By Propositions~\ref{prop3:reg} and \ref{lemma3:MalliavinXhk}, these
  are bounded independently of $h,k \in (0,1]$.
\end{proof}

\begin{lemma}
  \label{lemma4:stochconv}
  Let Assumptions~\ref{as1:A} and \ref{as1:Scheme} hold with $\beta\in(0,1]$.
  For $\gamma\in(0,\beta)$, set $p=\tfrac2{1-\gamma}$. It holds
  \begin{align*}
    \max_{n \in \{1,\ldots,N\}}
    \big\|W^A(t_n)-W_{h,k}^{A_h,n}\big\|_{\mathbf{M}^{1,p,p}(H)^\ast} \leq
    C\big(h^{2\gamma}+k^{\gamma}\big),\quad h,k\in(0,1].
  \end{align*}
\end{lemma}

\begin{proof}
  By \eqref{eq4:diffSC}, Theorem~\ref{lemma2:dual3}, and
  Assumption~\ref{as1:A} (iii), we get
\begin{equation*}
  \begin{split}
    \big\| W^A(t_n) - W_{h,k}^{A_h,n} \big\|_{\mathbf{M}^{1,p,p}(H)^*} &\leq
    \Big(\int_0^{t_n} \big\| \tilde{E}_{h,k}(t_n -t) \big\|_{\LB_2^0}^{p'}
    \diff{t}\Big)^\frac1{p'}\\
    &\leq \Big(\int_0^{t_n} \big\| \tilde{E}_{h,k}(t_n - t)
    A^{\frac{1-\beta}{2}} \big\|_{\LB}^{p'}
    \big\|A^{\frac{\beta-1}2}\big\|_{\LB_2^0}^{p'}
    \diff{t}\Big)^\frac1{p'}.
  \end{split}
\end{equation*}
Recalling the error operator \eqref{def1:ErrorOp} we obtain for $t \in
(t_{j}, t_{j+1})$, $j = 0,\ldots,n-1$,
\begin{align}
  \label{eq4:ErrorOneStep}
  \begin{split}
    &\big\| \tilde{E}_{h,k}(t_n-t) A^{\frac{1-\beta}{2}} \big\|_{\LB} \\
    &\quad \le \big\| \big( S(t_n-t) - S(t_n - t_j) \big) A^{\frac{1-\beta}{2}}
    \big\|_{\LB} + \big\| E_{h,k}^{n - j} A^{\frac{1-\beta}{2}}
    \big\|_{\LB} \\
    &\quad \le \big\| \big( I -  S(t-t_j) \big) A^{-\gamma}\big\|_{\LB}
    \big\| S(t_n - t) A^{\frac{2 \gamma + 1-\beta}{2}} \big\|_{\LB} +
    \big\| E_{h,k}^{n - j} A^{\frac{1-\beta}{2}} \big\|_{\LB}\\
    &\quad
\lesssim
 (t-t_j)^{\gamma}
 (t_n - t)^{-\frac{2\gamma + 1 - \beta}{2}}
+
\big( h^{2\gamma} + k^{\gamma} \big)
(t_n - t_j)^{-\frac{2\gamma + 1 - \beta}{2}}
\\ & \quad \lesssim
\big( h^{2\gamma} + k^{\gamma} \big)
(t_n - t)^{-\frac{2\gamma + 1 - \beta}{2}} ,
  \end{split}
\end{align}
where we applied \eqref{as1:SAnalytic} with $\varrho=\gamma$ and
\eqref{as1:SAnalytic2}, \eqref{as1:ErrorOneStep} with
$\theta=2\gamma$, $\varrho=1-\beta$. By recalling
\eqref{calc3:exponent}, we conclude
\begin{align*}
  \begin{split}
    \big\| W^A(t_n)-W_{h,k}^{A_h,n}
    \big\|_{\mathbf{M}^{1,p,p}(H)^*}
    &\lesssim \big( h^{2 \gamma} + k^{\gamma} \big)
    \Big( \int_{0}^{t_n} (t_n - t)^{-p' \frac{2\gamma + 1 - \beta}{2}} \diff{t}
    \Big)^{\frac{1}{p'}}\\
    &\lesssim\, h^{2\gamma}+k^{\gamma},
  \end{split}
\end{align*}
which is the desired result.
\end{proof}

\subsection{Semilinear equation with additive noise}
\label{subsec4:2}
Above we demonstrated that $V=\mathbf{M}^{1,p,p}(H)$ with $p$ large is
suitable for the weak error analysis for the stochastic
convolution. In order to treat semilinear equations we need a smaller
space. Here we work with the Gelfand triple
\begin{align*}
\mathbf{G}^{1,p}(H)\subset L^2(\Omega,H)\subset \mathbf{G}^{1,p}(H)^*.
\end{align*}
The line of proof is the same as above only that the convergence in
the dual norm is more involved and relies on the local Lipschitz
condition stated in Lemma~\ref{lemma5:Lipschitz}, the Burkholder type
inequality Lemma~\ref{lemma2:dual3} and a classical Gronwall argument.
\begin{theorem}
  \label{thm1:main}
  Let Assumptions~\ref{as1:A} and \ref{as1:Scheme} hold with $\beta
  \in (0,1]$.  Let $X$ and $X_{h,k}$ be the solutions to equations
  \eqref{eq1:SPDEmild} and \eqref{eq1:SPDEiterat}, respectively. For
  every function $\varphi \colon H \to H$ that satisfies
  Assumption~\ref{as1:phi} and every $\gamma\in[0,\beta)$, we have for
  $h,k\in(0,1]$ the weak convergence
  \begin{align*}
    &\max_{n \in \{1, \ldots,N\}}
    \big|\E\big[\varphi(X(t_n))-\varphi(X_{h,k}^n)\big]\big|\leq
    C\big(h^{2\gamma}+k^\gamma\big).
\end{align*}
\end{theorem}
\begin{proof}
  This is a direct consequence of \eqref{proc4:wc} and
  Lemmas~\ref{lemma4:linearization} and \ref{lemma3:Strong2} below.
\end{proof}

\begin{lemma}
  \label{lemma4:linearization}
  Let the assumptions of Theorem~\ref{thm1:main} hold. For
  $\gamma\in(0,\beta)$, set $p=\frac2{1-\gamma}$. It holds
  \begin{align*}
  \max_{n \in \{1,\ldots,N\}} \sup_{h,k \in (0,1]}
  \big\|\Phi_{h,k}^n\big\|_{\mathbf{G}^{1,p}(H)}\leq C.
  \end{align*}
\end{lemma}
\begin{proof}
  By Lemma~\ref{lemma4:Phi} we have
  $\|\Phi_{h,k}^n\|_{\mathbf{M}^{1,p,p}(H)}\leq C$ uniformly in $n$
  and $h,k$.  In addition, by \eqref{ineq2:moment},
  Proposition~\ref{lemma3:numstab}, and Assumption~\ref{as1:phi}, it
  holds $\|\Phi_{h,k}^n\|_{L^{2p}(\Omega,H)}\leq C$ uniformly in $n$
  and $h,k$.
\end{proof}

\begin{lemma}
  \label{lemma3:Strong2}
  Let the assumptions of Theorem~\ref{thm1:main} hold. For
  $\gamma\in(0,\beta)$, set $p=\frac2{1-\gamma}$. Then there exists a
  constant $C$ independent of $h, k \in (0,1]$ such that
  \begin{align*}
    &\max_{n \in \{1,\ldots,N\}}
    \big\|X(t_n)-X_{h,k}^n \big\|_{\mathbf{G}^{1,p}(H)^\ast} \le
    C \big(h^{2\gamma}+k^{\gamma}\big),\quad
    h,k\in(0,1].
  \end{align*}
\end{lemma}
\begin{proof}
  Let $n \in \{1,\ldots,N\}$ be arbitrary. By \eqref{eq1:SPDEmild} and
  \eqref{eq3:SPDEfulldisc}, we can write
  \begin{align*}
    \begin{split}
      &X(t_n)-X_{{h},k}^n= \big(S(t_n)-S_{h,k}^n \big)X_0\\
      &\quad+ \sum_{j = 0}^{n - 1} \int_{t_{j}}^{t_{j+1}}
      \big(S(t_n-t)- S_{{h},k}^{n-j} \big) F(X(t)) \diff{t}\\
      &\quad+ \sum_{j = 0}^{n - 1} \int_{t_j}^{t_{j+1}}
      S_{{h},k}^{n-j}  \big(F(X(t))-F( X_{{h},k}^j) \big) \diff{t}
      + W^A(t_n)-W_{h,k}^{A_h,n}.
    \end{split}
  \end{align*}
  By recalling the error operators ${E}_{h,k}^n$ from
  \eqref{def1:ErrorOp} and $\tilde{E}_{h,k}(t)$ from
  \eqref{eq4:ErrorOp}, we obtain
  \begin{align}
    \label{calc5:StrongError}
    \begin{split}
      &\big\|X(t_n)-X_{{h},k}^n\big\|_{\mathbf{G}^{1,p}(H)^\ast}\leq
      \big\| E_{h,k}^n X_0\big\|\\
      &\quad+\Big\|\int_0^{t_n}
      \tilde{E}_{{h},k}(t_n-t)F(X(t))\diff{t}\Big\|_{\mathbf{G}^{1,p}(H)^\ast}\\
      &\quad+\Big\| \sum_{j = 0}^{n - 1} \int_{t_j}^{t_{j+1}} S_{{h},k}^{n-j}
      \big(F(X(t))-F( X_{{h},k}^j )\big) \diff{t}
      \Big\|_{\mathbf{G}^{1,p}(H)^\ast}\\
      & \quad + \big\|W^A(t_n)-W_{h,k}^{A_h,n}\big\|_{\mathbf{G}^{1,p}(H)^\ast}.
    \end{split}
  \end{align}
  By \eqref{as1:ErrorOneStep} with $\varrho=-\theta=-2\gamma$ and
  Assumption~\ref{as1:A} (ii) we get
  \begin{align*}
    \big\| E_{h,k}^n  X_0\big\| \le
    \big\|E_{h,k}^n A^{-\gamma} \big\|_{\LB}
    \big\|A^\gamma X_0\big\|\lesssim \big(h^{2\gamma}+k^{\gamma}\big)
    \big\|A^\gamma X_0\big\|.
  \end{align*}
  For the second term in \eqref{calc5:StrongError} we first use that $\| Z
  \|_{\mathbf{G}^{1,p}(H)^\ast} \le \| Z \|_{L^2(\Omega,H)}$ for all $Z \in
  L^2(\Omega,H)$. Then by \eqref{eq4:ErrorOneStep} with $\beta = 1$, the linear
  growth of $F$, and \eqref{ineq2:moment} we have
  \begin{align*}
    &\Big\|\int_0^{t_n}
    \tilde{E}_{h,k}(t_n-t)F(X(t))\diff{t}\Big\|_{\mathbf{G}^{1,p}(H)^\ast}\leq
    \int_0^{t_n}
    \big\|\tilde{E}_{h,k}(t_n-t)\big\|_{\LB}\,\big\|F(X(t))\big\|_{L^2(\Omega,H)}
    \diff{t}\\
    &\qquad\lesssim \big(h^{2\gamma}+k^{\gamma}\big)
    \int_0^{t_n}(t_n-t)^{-\gamma}\diff{t}\,
    \Big(1+\sup_{t\in[0,T]}\big\|X(t)\big\|_{L^2(\Omega,H)}\Big)
    \lesssim h^{2\gamma}+k^{\gamma}.
  \end{align*}
  For the third summand we first notice that
  Propositions~\ref{prop3:reg} and \ref{lemma3:MalliavinXhk} justify
  the use of Lemma~\ref{lemma5:Lipschitz} with $\eta=F$, $U=H$,
  $V=\dot{H}^{-1}$, $X_1=X(t)$ and $X_2=X_{h,k}^j$ with $t \in
  (t_j,t_{j+1}]$. We get
  \begin{align*}
    \begin{split}
    &\big\| F(X(t))-F( X_{h,k}^j)
    \big\|_{\mathbf{G}^{1,p}(\dot{H}^{-1})^\ast} \le \max_{i \in \{1,2\}}
    |F|_{\Cb^i(H,\dot{H}^{-1})}\\
    &\qquad\times\Big(1 + \| X(t) \|_{\mathbf{M}^{1,2p,p}(H)} + \| X_{h,k}^j
    \|_{\mathbf{M}^{1,2p,p}(H)} \Big) \| X(t) - X_{h,k}^j
    \|_{\mathbf{G}^{1,p}(H)^\ast}\\
    &\quad \lesssim \| X(t) - X_{h,k}^j \|_{\mathbf{G}^{1,p}(H)^\ast}.
    \end{split}
  \end{align*}
  By \eqref{as1:Analyt}, \eqref{as1:AhPhA} with $\rho = \frac{1}{2}$, and 
  Lemma~\ref{lemma3:ideal},
  we get for the third term
  \begin{align}
    \label{eq5:intT3}
    \begin{split}
      &\Big\| \sum_{j = 0}^{n - 1} \int_{t_j}^{t_{j+1}} S_{{h},k}^{n-j}
      A_h^{\frac12}A_h^{-\frac12}P_hA^{\frac12}A^{-\frac12}\big(F(X(t))-F(
      X_{h,k}^{j})\big)\diff{t}
      \Big\|_{\mathbf{G}^{1,p}(H)^\ast}\\
      &\quad \le \sum_{j = 0}^{n-1} \int_{t_j}^{t_{j+1}}
      \big\|
      S_{{h},k}^{n-j}A_h^{\frac12}\big\|_{\LB}
      \|A_h^{-\frac12}P_hA^{\frac12}\|_{\LB} \big\| F(X(t))-F(
      X_{h,k}^{j}) \big\|_{\mathbf{G}^{1,p}(\dot{H}^{-1})^\ast} \diff{t}\\
      &\quad \lesssim \sum_{j = 0}^{n-1} \int_{t_j}^{t_{j+1}}
      t_{n-j}^{-\frac12}  \big(
      \big\| X(t)-X(t_j) \big\|_{\mathbf{G}^{1,p}(H)^\ast}
      +\big\| X(t_j)-X_{h,k}^j\big\|_{\mathbf{G}^{1,p}(H)^\ast} \big)
      \diff{t}.
    \end{split}
  \end{align}
  By Proposition~\ref{lemma4:tempreg}, it holds
  $\|X(t)-X(t_j)\|_{\mathbf{G}^{1,p}(H)^\ast}\lesssim k^\gamma$ and
  therefore
  \begin{align*}
    \begin{split}
      &\Big\| \sum_{j = 0}^{n - 1} \int_{t_j}^{t_{j+1}} S_{{h},k}^{n-j}
      \big(F(X(t))-F( X_{h,k}^{j})\big)\diff{t}
      \Big\|_{\mathbf{G}^{1,p}(H)^\ast}\\
      &\qquad \lesssim k^{1+\gamma}\sum_{j = 0}^{n-1}t_{n-j}^{-\frac12}
      +k\sum_{j = 0}^{n-1} t_{n-j}^{-\frac12}\big\|
      X(t_j)-X_{h,k}^j\big\|_{\mathbf{G}^{1,p}(H)^\ast}.
    \end{split}
  \end{align*}
  The fourth summand is estimated in
  Lemma~\ref{lemma4:stochconv}. Altogether we conclude that
  \begin{align*}
    \big\| X(t_n)-X_{{h},k}^n \big\|_{\mathbf{G}^{1,p}(H)^\ast} \lesssim
    \big( h^{2\gamma}+k^{\gamma} \big) +k\sum_{j=0}^{n-1}t_{n-j}^{-\frac12}
    \big\|X(t_j)-X_{{h},k}^j\big\|_{\mathbf{G}^{1,p}(H)^\ast}.
  \end{align*}
  By the discrete Gronwall Lemma~\ref{lemma2:Gronwall} the assertion follows.
\end{proof}

Weak approximation concerns the approximation of the Markov semigroup.
In view of Theorem~\ref{thm1:main} and Corollary~\ref{cor:Markov}, we see
that the rate of weak convergence in time coincides with the H\"{o}lder
regularity in time for the Markov semigroup, which is intuitively to be
expected for an Euler approximation. A similar
connection to the discretization in space seems to be a more subtle issue.

The relationship between the strong and weak rate of convergence can
also be seen in the view of duality. The following corollary deduces a
strong convergence result from Lemma~\ref{lemma3:Strong2} and
Propositions~\ref{prop3:reg} and \ref{lemma3:MalliavinXhk}. It
indicates why one often encounters the rule of thumb that the order of
weak convergence is twice the order of strong convergence.

\begin{corollary}
  \label{cor4:strong}
  Let the assumptions of Theorem~\ref{thm1:main} hold. Let $X$ and
  $X_{h,k}$ denote the solutions to equations \eqref{eq1:SPDEmild} and
  \eqref{eq1:SPDEiterat}, respectively. Then for every
  $\gamma\in(0,\beta)$ there exists a constant $C$ such that
  \begin{align*}
    \max_{n \in \{1,\ldots,N\}}\|X(t_n)-X_{h,k}^n\|_{L^2(\Omega,H)}
    \leq C(h^\gamma+k^\frac\gamma2),\quad h,k\in(0,1].
  \end{align*}
\end{corollary}

\begin{proof}
  For arbitrary $n \in \{1,\ldots,N\}$ we have by the duality argument
  with $p = \frac{2}{1 - \gamma}$
  \begin{align*}
    &\|X(t_n)-X_{h,k}^n\|_{L^2(\Omega,H)}^2 = \big\langle X(t_n)-X_{h,k}^n,
    X(t_n)-X_{h,k}^n \big\rangle_{L^2(\Omega,H)}\\
    &\quad \le \big( \| X(t_n) \|_{\mathbf{G}^{1,p}(H)} + \| X_{h,k}^n
    \|_{\mathbf{G}^{1,p}(H)} \big) \|
    X(t_n)-X_{h,k}^n \|_{\mathbf{G}^{1,p}(H)^\ast}.
  \end{align*}
  The first factor is bounded independently of $n\in \{1,\ldots,N\}$,
  by Propositions~\ref{prop3:reg} and \ref{lemma3:MalliavinXhk}. For
  the second factor we apply Lemma~\ref{lemma3:Strong2} and since
  $(h^{2\gamma} + k^\gamma)^{\frac{1}{2}} \le
  (h^\gamma+k^\frac\gamma2)$ for all $h,k \in (0,1]$ the result
  follows.
\end{proof}

\subsection{Multiplicative noise}
\label{subsec4:3}
The choice $V=\mathbf{G}^{1,p}(H)$ of Subsection~\ref{subsec4:2} works
only for equations with additive noise.  We demonstrate this here by
considering the following equation with linear multiplicative noise
\begin{align*}
  \diff{X}(t)+AX(t)\diff{t}
  =B X(t)\diff{W(t)},\; t\in(0,T];\quad
  X(0)=X_0.
\end{align*}
Here $B\in \LB(H,\LB_2(H_0,\dot{H}^{\beta-1}))$. In order to perform
the Gronwall argument in the $\mathbf{G}^{1,p}(H)^\ast$-norm for this
equation, one would need a bound
\begin{equation}
\label{eq:mult_bound}
\begin{split}
      &\Big\| \sum_{j = 0}^{n - 1} \int_{t_j}^{t_{j+1}} S_{{h},k}^{n-j}
      B\big(X(t)- X_{h,k}^{j}\big)\diff{W(t)}
      \Big\|_{\mathbf{G}^{1,p}(H)^\ast}\\
      &\qquad\lesssim \sum_{j = 0}^{n-1} \int_{t_j}^{t_{j+1}} \big\|
      X(t)-X_{h,k}^j
      \big\|_{\mathbf{G}^{1,p}(H)^\ast} \diff{t},
\end{split}
\end{equation}
cf.~\eqref{eq5:intT3}. Attempting to prove this, we integrate by parts
and move the supremum inside the integral to get
\begin{equation*}
  \begin{split}
    &\Big\| \sum_{j = 0}^{n - 1} \int_{t_j}^{t_{j+1}} S_{{h},k}^{n-j}
    B\big(X(t)- X_{h,k}^{j}\big)\diff{W(t)}
    \Big\|_{\mathbf{G}^{1,p}(H)^\ast}\\
    & \quad =\sup_{Z\in
      \mathbf{G}^{1,p}(H)}\frac1{\|Z\|_{\mathbf{G}^{1,p}(H)}}
\Big\langle Z,
\sum_{j = 0}^{n - 1}
    \int_{t_j}^{t_{j+1}} S_{{h},k}^{n-j}
    B\big(X(t)- X_{h,k}^{j}\big)\diff{W(t)}\Big\rangle_{L^2(\Omega,H)}\\
    & \quad \leq \sum_{j = 0}^{n - 1}\int_{t_j}^{t_{j+1}}\sup_{Z\in
      \mathbf{G}^{1,p}(H)}\frac1{\|Z\|_{\mathbf{G}^{1,p}(H)}}\big\langle
    B^*S_{h,k}^{n-j}D_tZ, X(t)-X_{h,k}^{j} 
    \big\rangle_{L^2(\Omega,H)}\diff{t}.
  \end{split}
\end{equation*}
If it would hold $B^*S_{h,k}^{n-j}D_t\in\LB(\mathbf{G}^{1,p}(H))$, then
the bound \eqref{eq:mult_bound} would follow, but this is
not the case as only $D_t\colon \mathbf{G}^{1,p}(H)\to
L^p(\Omega,\LB_2^0)$ for {a.e.} $t\in[0,T]$. We see no other natural
choice of the space $V$ but it might be that the estimate
\eqref{ineq:dualbound} is too crude in order to treat multiplicative
noise.

\section{Approximation by the finite element method}
\label{sec5}
In this section we describe an explicit example for the linear
operator $A$ and its corresponding numerical discretization by the
finite element method.

For this we consider the Hilbert space $H=L^2(D)$, where
$D\subset\R^d$, $d=1,2,3$, is a bounded, convex, and polygonal
domain. The linear operator $(A,\D(A))$ is defined to be
$Au=-\nabla\cdot(a\nabla u)+cu$ with Dirichlet boundary conditions,
where $a, c \colon D \to \R$ are sufficiently smooth with $c(\xi) \ge
0$ and $a(\xi) \ge a_0 > 0$ for $\xi\in D$. Then $A$ is an elliptic,
selfadjoint, second order differential operator with compact inverse,
see for instance \cite{evans}. In particular, $A$
satisfies Assumption \ref{as1:A} (i).

We measure spatial regularity in terms of the abstract spaces
$\dot{H}^\theta$, $\theta \in \R$, which now are related to the
classical Sobolev spaces, for example $\dot{H}^1 = H_0^1(D)$ and
$\dot{H}^2 = H_0^1(D) \cap H^2(D)$. For more details we refer to
\cite{kruse2013}*{App.~B.2} and the references therein.

Let $(T_h)_{h\in(0,1]}$ be a regular family of triangulations of $D$
with maximal mesh size $h\in(0,1]$. We define a family of subspaces
$(V_h)_{h\in(0,1]}$ of $\dot{H}^1$, consisting of continuous piecewise
linear functions corresponding to $(T_h)_{h\in(0,1]}$. By equipping
the space $\dot{H}^1$ with the inner product
$\langle\cdot,\cdot\rangle_1:=\langle
A^\frac12\cdot,A^\frac12\cdot\rangle$, we define $A_h\colon
V_h\rightarrow V_h$, $h \in (0,1]$, to be the linear operators given
by
\begin{align*}
  \langle A_h v_h,u_h\rangle=\langle v_h,u_h\rangle_1,\quad\forall v_h,u_h\in
  V_h.
\end{align*}
Now, from \cite{kruse2013}*{(3.15)} we get $\| A_h^{-1} P_h x \| \le \|
x\|_{-1}$ for all $x \in \dot{H}^{-1}$. Hence, it holds
\begin{align*}
  \| A_h^{-\frac{1}{2}} P_h A^{\frac{1}{2}} \|_{\LB} \le 1.
\end{align*}
An interpolation between this and $\| P_h \|_{\LB} \le 1$ yields
\eqref{as1:AhPhA} for $\varrho \in [0, 1]$.

As in Subsection~\ref{subsec2:3} we denote by $(S(t))_{t\geq0}$ the
semigroup generated by $-A$ and $S_{h,k} := (I + k A_h)^{-1} P_h$. The
standard literature on finite element methods, for instance
\cite{thomee2006}, provides error estimates for the approximation of
the semigroup with smooth and nonsmooth initial data. More precisely,
it holds for the error operator \eqref{eq4:ErrorOp} that
\begin{align*}
  \| \tilde{E}_{h,k}(t) x\| \leq C
  \big(h^2+k\big)t^{-\frac{2-q}2} \|x\|_{\dot{H}^q}, \quad x\in\dot{H}^q,\
  q=0,2.
\end{align*}
By interpolation this covers the smooth data case
$-\theta\le\varrho\le0$ of \eqref{as1:ErrorOneStep}. For the purpose
of the present work we need to extend this to less regular initial
data. This is done by the next lemma, which is a consequence of
\cite{kruse2013}*{Lemma~3.12}.

\begin{lemma} \label{lemma:femett} Under the above assumptions and for
  $0\leq\theta\leq 2$ and
  $\stig{-\theta}\leq\varrho\leq\min(1,2-\theta)$, the following
  estimate holds true
  \begin{align*}
    \|\tilde{E}_{h,k}(t)x\|&\leq
    C\big(h^{\theta}+k^{\frac\theta2}\big)t^{-\frac{\theta+\varrho}2}
    \|x\|_{-\varrho},\quad x\in\dot{H}^{-\varrho},\ t>0,\ h,k\in(0,1].
  \end{align*}
\end{lemma}

\begin{proof} As noted above it remains to treat the case when
    $0\leq\varrho\leq\min(1,2-\theta)$.  By
  \cite{kruse2013}*{Lemma~3.12 (i)} the estimate
\begin{align}
  \label{ineq5:interp1}
  \|\tilde{E}_{h,k}(t)x\|\leq
  C\big(h^{\theta}+k^{\frac\theta2}\big)t^{-\frac\theta2}\|x\|,\quad t>0,\
  0\leq\theta\leq2,
\end{align}
holds for all $h,k \in (0,1]$. By \cite{kruse2013}*{Lemma~3.12 (iii)} the error
operator $\tilde{E}_{h,k}$ also satisfies, for $1 \le \theta \le 2$,
\begin{align}
  \label{ineq5:interp2}
  \|\Tilde{E}_{h,k}(t)x\|\leq
  C\big(h^{\theta}+k^{\frac\theta2}\big)t^{-1}\|x\|_{-(2-\theta)},\quad t>0.
\end{align}
Interpolation of \eqref{ineq5:interp1} and \eqref{ineq5:interp2} with fixed
$\theta \in[1,2]$ gives that, for $\lambda \in [0,1]$,
\begin{align*}
  \|\Tilde{E}_{h,k}(t)x\|&\leq
  C\big(h^{\theta}+k^{\frac\theta2}\big)t^{-(1-\lambda)\frac\theta2}t^{-\lambda}
  \|x\|_{-\lambda(2-\theta)}\\
  &= C\big(h^{\theta}+k^{\frac\theta2}\big)
  t^{-\frac\theta2-\frac{\lambda(2-\theta)}2} \|x\|_{-\lambda(2-\theta)},\quad
  t>0.
\end{align*}
If we let $\varrho =\lambda(2-\theta)$, then we get the following
estimate: for $1\leq\theta\leq2$ and $0\leq\varrho\leq 2-\theta$,
\begin{align}
  \label{ineq5:interp3}
  \|\Tilde{E}_{h,k}(t)x\|&\leq
  C\big(h^{\theta}+k^{\frac\theta2}\big)t^{-\frac{\theta+\varrho}2}
  \|x\|_{{-\varrho}},\quad t\geq0.
\end{align}
By \cite{kruse2013}*{Lemma~3.12 (ii)} it holds
\begin{align}
  \label{ineq5:interp4}
  \|\Tilde{E}_{h,k}(t)x\|\leq Ct^{-\frac\varrho2}\|x\|_{{-\varrho}},\quad t>0,\
  0\leq\varrho\leq1,
\end{align}
and using \eqref{ineq5:interp3} with $\theta=1$ and \eqref{ineq5:interp4}, both
with the same $0\leq\varrho\leq1$, yields
\begin{equation}
\begin{split}
\label{ineq5:interp5}
\|\Tilde{E}_{h,k}(t)x\|&=\|\Tilde{E}_{h,k}(t)x\|^\lambda
\|\Tilde{E}_{h,k}(t)x\|^{1-\lambda}\leq C\big(h+k^{\frac12}\big)^\lambda
t^{-\frac{\lambda+\varrho}2}\|x\|_{{-\varrho}}\\
&\leq C\big(h^{\lambda}+k^{\frac\lambda2}\big)t^{-\frac{\lambda+\varrho}2}
\|x\|_{{-\varrho}},\quad t>0,\ 0\leq\lambda\leq1.
\end{split}
\end{equation}
Combining \eqref{ineq5:interp3} and \eqref{ineq5:interp5} concludes the proof.
\end{proof}

Writing the statement of the lemma in operator form yields
\begin{align*}
  \|\tilde{E}_{h,k}(t)A^\frac\varrho2\|_{\LB}&\leq
  C\big(h^{\theta}+k^{\frac\theta2}\big)t^{-\frac{\theta+\varrho}2},\quad t>0,\
  0\leq\theta\leq 2,\  \stig{-\theta}\leq \varrho\leq \min(1,2-\theta).
\end{align*}
This is \eqref{as1:ErrorOneStep} for the finite element
method. To verify Assumption~\ref{as1:Scheme} it remains to show
\eqref{as1:Analyt}. By \cite{kruse2013}*{(3.42)}
\begin{align*}
  \| S_{h,k}^n x\|\leq C t^{-\frac12} \|x\|_{-1}.
\end{align*}
Interpolating between this and $\| S_{h,k}^n x\| \leq C\|x\|$ yields
\eqref{as1:Analyt}.

\subsection*{Acknowledgement}
The authors wish to thank M.~Kov\'{a}cs for fruitful discussions
during the preparation of the work \cite{AnderssonKovacsLarsson},
which led to improvements of the present paper.  We also thank
A.~Lang and X.~Wang for valuable comments on an earlier version of the
manuscript \adam{and A.~Jentzen for making us aware of a reference}.

The first two authors also acknowledge the kind support by W.-J.~Beyn,
B.~Gentz, and the DFG-funded CRC 701 'Spectral Structures and
Topological Methods in Mathematics' by making possible an inspiring
research stay at Bielefeld University, where part of this work was
written.

\def\cprime{$'$} \def\polhk#1{\setbox0=\hbox{#1}{\ooalign{\hidewidth
  \lower1.5ex\hbox{`}\hidewidth\crcr\unhbox0}}}


\end{document}